\documentclass[]{amsart}
\usepackage[utf8]{inputenc}

\usepackage[maxnames=10]{biblatex}
\usepackage{amsmath}
\usepackage{amssymb}
\usepackage{amsthm}
\usepackage{stmaryrd}
\usepackage{hyperref}
\usepackage{cleveref}
\usepackage{listings}

\usepackage[hyphenbreaks]{breakurl}

\usepackage{xcolor}
\usepackage{dutchcal}
\usepackage{tikz-cd}
\usetikzlibrary{calc,decorations.markings,decorations.pathmorphing}
\usepackage{ifthen}
\usepackage{bm}

\newcommand*\circled[1]{\tikz[baseline=(char.base)]{
            \node[shape=circle,draw,inner sep=1pt] (char) {#1};}}

\newcommand{\ULozenge}[1]{
\draw #1-- ($#1 - (Z)$) -- ($#1 + (X)$)-- ($#1 - (Y) $)--cycle;
}

\newcommand{\fULozenge}[1]{
\filldraw [fill=red] #1-- ($#1 - (Z)$) -- ($#1 + (X)$)-- ($#1 - (Y) $)--cycle;
}

\newcommand{\VLozenge}[1]{
\draw #1 -- ($#1 + (X)$) -- ($#1 + (X) + (Z)$) -- ($#1 + (Z)$) -- cycle;
}
\newcommand{\fVLozenge}[1]{
\filldraw [fill=yellow] #1 -- ($#1 + (X)$) -- ($#1 + (X) + (Z)$) -- ($#1 + (Z)$) -- cycle;
}

\newcommand{\WLozenge}[1]{
\draw #1 -- ($#1 + (X)$) -- ($#1 + (X) - (Y)$)-- ($#1 - (Y)$) -- cycle;
}

\newcommand{\fWLozenge}[1]{
\filldraw [fill=green] #1 -- ($#1 + (X)$) -- ($#1 + (X) - (Y)$)-- ($#1 - (Y)$) -- cycle;
}

\newcommand{\fCube}[1]{
\fVLozenge{#1};
\fULozenge{($#1 - (X)$) };
\fWLozenge{($#1 + (Y)$)};
}

\theoremstyle{definition}
\newtheorem{Def}[]{Definition}[section]
\newtheorem{Pro}[Def]{Proposition}
\newtheorem{Rem}[Def]{Remark}
\newtheorem{Lem}[Def]{Lemma}
\newtheorem{Exp}[Def]{Example}
\newtheorem{Theo}[Def]{Theorem}
\newtheorem{Cor}[Def]{Corollary}
\newtheorem{Ques}[Def]{Question}
\newtheorem{Con}[Def]{Construction}

\DeclareMathOperator{\Ext}{Ext}
\DeclareMathOperator{\Hom}{Hom}
\DeclareMathOperator{\RHom}{\mathbb{R}\strut\kern-.2em\operatorname{Hom}}

\DeclareMathOperator{\Ker}{Ker}

\DeclareMathOperator{\SL}{SL}
\DeclareMathOperator{\GL}{GL}
\DeclareMathOperator{\gdim}{gdim}
\DeclareMathOperator{\supp}{supp}
\DeclareMathOperator{\Db}{\mathrm{D}^{\mathrm{b}}}

\DeclareMathOperator{\diag}{diag}
\DeclareMathOperator{\chr}{char}

\newcommand{\bbone}{\text{\usefont{U}{bbold}{m}{n}1}}
\MakeRobust{\bbone}

\renewcommand{\restriction}{\vert}
\renewcommand{\mod}{\operatorname{mod}}
\renewcommand{\Bbbk}{k}
\renewcommand{\Im}{\operatorname{Im}}

\title{The $3$-preprojective algebras of type $\Tilde{A}$}
\author{Darius Dramburg}
\address{Darius Dramburg, Department of Mathematics, Uppsala University, Box 480, 751 06 Uppsala, Sweden}
\email{darius.dramburg@math.uu.se}

\author{Oleksandra Gasanova}
\address{Oleksandra Gasanova, Faculty of Mathematics, University of Duisburg-Essen}
\email{oleksandra.gasanova@uni-due.de}

\date{\today}

\begin{document}

\begin{abstract}
Let $G \leq \SL_{n+1}(\mathbb{C})$ act on $R = \mathbb{C}[X_1, \ldots, X_{n+1}]$ by change of variables. Then, the skew-group algebra $R \ast G$ is bimodule $(n+1)$-Calabi-Yau. In certain circumstances, this algebra admits a locally finite-dimensional grading of Gorenstein parameter $1$, in which case it is the $(n+1)$-preprojective algebra of its $n$-representation infinite degree $0$ piece, as defined in \cite{HIO}. If the group $G$ is abelian, the $(n+1)$-preprojective algebra is said to be of type $\Tilde{A}$. For a given group $G$, it is not obvious whether $R \ast G$ admits such a grading making it into an $(n+1)$-preprojective algebra. We study the case when $n=2$ and $G$ is abelian. We give an explicit classification of groups such that $R \ast G$ is $3$-preprojective by constructing such gradings. This is possible as long as $G$ is not a subgroup of $\SL_2(\mathbb{C})$ and not $C_2 \times C_2$. For a fixed $G$, the algebra $R \ast G$ admits different $3$-preprojective gradings, so we associate a type to a grading and classify all types. Then we show that gradings of the same type are related by a certain kind of mutation. This gives a classification of $2$-representation infinite algebras of type $\tilde{A}$. The involved quivers are those arising from hexagonal dimer models on the torus, and the gradings we consider correspond to perfect matchings on the dimer, or equivalently to periodic lozenge tilings of the plane. Consequently, we classify these tilings up to flips, which correspond to the mutation we consider. 
\end{abstract}

\maketitle

\section*{Introduction}
Preprojective algebras were defined in \cite{gel1979model} by Gelfand and Ponomarev in the context of representation theory of hereditary algebras. Since then, they have been shown to arise naturally in many areas of mathematics, one famous case being the so-called McKay correspondence, see \cite{MR0604577}. Under this correspondence, finite subgroups of $G \leq \SL_2(\mathbb{C})$ are identified with simply laced Dynkin diagrams, which in turn correspond to the resolution graphs of the quotient singularities $\mathbb{C}^2/G$. A representation-theoretic interpretation of this fact is given by a Morita equivalence between the skew-group algebra $\Bbbk[x,y] \ast G$ and the preprojective algebra $\Pi(\overline{Q})$ of the corresponding extended Dynkin diagram. A reinterpretation of the preprojective algebra due to Baer, Geigle and Lenzing in \cite{BGL} states that one can introduce a grading on the preprojective algebra so that the degree $0$ piece $\Lambda = \Pi(\overline{Q})_0$ is a finite-dimensional hereditary representation infinite algebra $\Lambda = \Bbbk Q$, and such that the preprojective algebra can be recovered as $\Pi(\overline{Q}) = T_\Lambda(\Ext^1_\Lambda(D(\Lambda), \Lambda))$. In this way, one obtains precisely the connected tame hereditary algebras. 

It is an interesting question how to generalise the above situation to higher dimension. On the geometric side, increasing dimension can be done by considering finite subgroups $G \leq \SL_{n+1}(\mathbb{C})$ for $n \geq 1$. On the preprojective algebra side, the generalisation was given in terms of higher-dimensional Auslander-Reiten theory. In \cite{HIO}, Herschend, Iyama and Opperman defined $n$-representation infinite algebras, generalising hereditary representation infinite algebras. In certain cases, their $(n+1)$-preprojective algebras then relate to a finite subgroup $G \leq \SL_{n+1}(\mathbb{C})$ in a way similar to the case $n=1$. Some other connections to (non-commutative) algebraic geometry arise from the fact that $n$-representation infinite algebras are quasi-Fano algebras as defined by Minamoto in \cite{minamoto2012ampleness}, as well as from the work in \cite{BuchweitzHille, MinamotoMori, HIMO}. Furthermore, higher preprojective algebras often lead to interesting triangle equivalences between stable categories of Cohen-Macaulay modules and generalised cluster categories, see \cite{AIR}. It is therefore of interest to describe higher preprojective algebras and corresponding $n$-representation infinite algebras. We refer the reader to \cite{JassoKvamme} for an introduction to higher-dimensional Auslander-Reiten theory. 

The counterpart to the $n$-representation infinite algebras in higher-dimensional Auslander-Reiten theory are called $n$-representation finite algebras, which have been studied extensively, e.g. in \cite{IyamaOppermann, IyamaOppermannStable, HigherNakayama, Herschend_Iyama_2011}, and large classes of examples are known. The $n$-representation finite algebras of type $A$ are classified in \cite{IyamaOppermann}. Fewer examples of $n$-representation infinite algebras have been computed. Currently, the main sources of $n$-representation infinite algebras are tensor product constructions, discussed in \cite{HIO}, the mentioned non-commutative geometry, dimer models, studied in \cite{Nakajima}, and skew-group algebras, such as in \cite{Giovannini}, as well as through higher APR-tilting as proven in \cite{MizunoYamaura}. In this article, we want to expand on the class of $n$-representation infinite algebras coming from skew-group algebras and study their higher APR-tilts. 

Deciding whether a given algebra is $n$-representation infinite is a non-trivial problem. A general strategy for computing $n$-representation infinite algebras therefore starts with their higher preprojective algebras. In detail, a finite-dimensional algebra $\Lambda$ is called $n$-representation infinite if $\gdim(\Lambda) \leq n$ and $ \nu_n^{-i}(\Lambda)$ is quasi-isomorphic to a complex concentrated in degree $0$ for all $i\geq 0$, where $\nu_n^{-1} = \nu^{-1} \circ [n]$ and $\nu$ is the derived Nakayama functor. The higher preprojective algebra is 
\[ \Pi_{n+1}(\Lambda) = T_\Lambda(\Ext^{n+1}_\Lambda(D(\Lambda), \Lambda)).\]
This algebra is bimodule $(n+1)$-Calabi-Yau, and as a tensor algebra it comes equipped with a natural grading. This grading is locally finite-dimensional and of Gorenstein parameter $1$, and the original $\Lambda$ can be recovered as its degree $0$ part. In fact, it is known from \cite{KellerVandenBergh, MinamotoMori, HIO} that the assignment $\Lambda \mapsto \Pi_{n+1}(\Lambda)$ gives a bijection between $n$-representation infinite algebras and bimodule $(n+1)$-Calabi-Yau algebras equipped with a locally finite-dimensional grading of Gorenstein parameter $1$, both sides taken up to isomorphism.
Note that the Calabi-Yau property is one of the ungraded algebra, while the Gorenstein parameter is a property of the grading. Thus, to construct $n$-representation infinite algebras, one can consider bimodule $(n+1)$-Calabi-Yau algebras and construct appropriate gradings. This strategy was employed in \cite{AIR, Giovannini} to produce examples and in \cite{Thibault} to prove that certain Calabi-Yau algebras do not arise as higher preprojective algebras. 

We follow this general strategy to give a classification of $2$-representation infinite algebras of type $\Tilde{A}$. This type was defined in \cite{HIO} as those $n$-representation infinite algebras $\Lambda$ for which the higher preprojective algebra is isomorphic to the polynomial ring $R = \Bbbk[X_1, \ldots, X_{n+1}]$ over an algebraically closed field $k$ of characteristic $0$, skewed by the action of an abelian group $G \leq \SL_{n+1}(k)$ acting on the variables, i.e. $\Pi_{n+1}(\Lambda) \simeq R \ast G$. However, as shown by Thibault in \cite{Thibault}, not every skew-group algebra does admit a grading that makes it into a higher preprojective algebra. Our classification for $n=2$ therefore consists of two steps. First, we classify the skew-group algebras which can be endowed with a grading making them into higher preprojective algebras. Then, for a fixed such algebra, we classify all possible gradings, up to isomorphism of the degree $0$ part. 

For the first step, we decompose the skew-group algebra $R \ast G = \Bbbk[X_1, X_2, X_3] \ast G$ into an iterated skew-group algebra $(R \ast H) \ast K$, construct a $3$-preprojective structure on $R \ast H$ using methods from \cite{AIR}, and then use methods from \cite{LeMeur} to extend this to a $3$-preprojective structure on $R \ast G$. This, together with a criterion given in \cite{Thibault} for $R \ast G$ not to be higher preprojective, yields our classification of $3$-preprojective algebras of type $\Tilde{A}$. We later give a second proof which does not rely on the mentioned results.

\begin{Theo}[= \Cref{Theo:SL3 Classification}]
Let $G \leq \SL_3(k)$ be abelian, acting on $R = \Bbbk[X_1, X_2, X_3]$ via change of variables. Then $R \ast G$ can be endowed with a higher preprojective grading if and only if $G$ is not isomorphic to $C_2 \times C_2$ and the embedding of $G$ into $\SL_3(k)$ does not factor through any embedding of $\SL_1(k) \times \SL_2(k)$ into $\SL_3(k)$.      
\end{Theo}

We believe that a similar approach is possible to classify all $(n+1)$-preprojective algebras of type $\Tilde{A}$, provided the case of a skew-group algebra $R \ast G$ by a cyclic group $G \leq \SL_{n+1}(k)$ is solved, cf.\ \Cref{Rem: General HPG strategy}. 
We also note a curious connection to singularity theory. Associated to the group $G \leq \SL_3(k)$, we have the quotient singularity given by the invariant ring $R^G$. Amongst the abelian groups $G$, those for which $R \ast G$ is not a higher preprojective algebra are precisely those whose quotient singularity is compound DuVal. In fact, it is easy to see from the quivers in \cite{NdCS} that other compound DuVal quotient singularities do not admit any $3$-preprojective gradings on their associated skew-group algebras. However, there are examples of quotient singularities which are not compound DuVal and whose skew-group algebras do not admit any $3$-preprojective gradings, see \cite[Example 5.5]{BSW}. Also, note that we make no assumption on $R^G$ having an isolated singularity and indeed we construct $3$-preprojective gradings for many algebras $R \ast G$ where $R^G$ does not have an isolated singularity. 

For the classification of $2$-representation infinite algebras, we then fix an algebra $R \ast G $ given by \Cref{Theo:SL3 Classification} and work explicitly with its quiver $Q$, defined by a matrix $B$ associated to the given presentation of $G$, cf.\ \Cref{Pro: Periodicity matrix}. The quiver may admit non-isomorphic $3$-preprojective gradings, called cuts. To each cut we associate a triple of numbers called its type. We prove that the types can be seen as certain lattice points in a simplex, and that cuts of the same type are related by a certain kind of mutation, cf.\ \Cref{Def: Mutation of cuts}. In our proofs, we make use of the theory of periodic lozenge tilings of the plane, which is dual to that of hexagonal dimer models on the torus. The theory of dimer models provides some statements in greater generality. In detail, Nakajima proved in \cite{Nakajima} that the internal lattice points of the perfect matching polygon of a consistent dimer model classify the possible $3$-preprojective cuts of the Jacobian algebra up to mutation of the dimer model. It is also proven that such an internal lattice point exists if the center of the Jacobian algebra has an isolated singularity (which is not $A_1$). Our version of these results reads as follows, where $n = |G|$ and $(\theta_1(C), \theta_2(C), \theta_3(C))$ denotes the type of a cut $C$. 

\begin{Theo}[= \Cref{Theo: Divisibility conditions}]
The type of any higher preprojective cut $C$ satisfies $\theta_1(C) + \theta_2(C) + \theta_3(C) = n $ and $\theta_i(C) > 0$ for $1 \leq i \leq 3$. Furthermore, a triple $(\gamma_1, \gamma_2, \gamma_3) \in \mathbb{N}_{> 0}^{1 \times 3}$ with $\gamma_1 + \gamma_2 + \gamma_3 = n$ is the type of a higher preprojective cut if and only if we have 
    \[ (\gamma_1, \gamma_2) B \in n \mathbb{Z}^{1 \times 2}. \]    
\end{Theo}

\begin{Theo}[= \Cref{Theo: Mutation is transitive}]
Two cuts $C_1$, $C_2$ have the same type if and only if they are connected by a mutation sequence, i.e. they fulfill $\theta(C_1) = \theta(C_2)$ if and only if there exists a sequence of vertices $v_1, \ldots, v_m \in Q_0$ such that $(\mu_{v_m} \circ \cdots \circ \mu_{v_1})(C_1)  = C_2$.    
\end{Theo}

A consequence of the explicit condition in \Cref{Theo: Divisibility conditions} is that we obtain a construction of the possible types for arbitrary abelian $G \leq \SL_3(k)$. This allows us to reprove \Cref{Theo:SL3 Classification} independently of \cite{AIR, LeMeur, Thibault}.

Our proof of \Cref{Theo: Divisibility conditions} is constructive, and we give a formula for constructing a periodic lozenge tiling, hence a perfect matching, for any type, cf.\ \Cref{tilingconstr}. The tiling we construct enjoys some special properties, which shows that any $2$-representation infinite algebra $\Lambda$ of type $\Tilde{A}$ is mutation-equivalent to a skew-group algebra $\Lambda' \ast C_m$, where $\Lambda'$ arises from the construction in \cite{AIR}, cf.\ \Cref{Cor: Tilting equiv to skewed AIR}. The authors of \cite{AIR} mention that it would be interesting to generalise their results to non-cyclic groups, so our statement makes the connection between abelian and cyclic groups explicit.

Our theorems on lozenge tilings are based on unpublished work \cite{Martin} by Ibeas Martín, but we employ different methods of proof and hence reprove the corresponding statements classifying periodic lozenge tilings of the plane. 

\subsection*{Outline}
In \Cref{Sec:skew-group algebras}, we recall the construction of a skew-group algebra and summarise an important construction in terms of derivation algebras of superpotentials. 

In \Cref{Sec: HPAs}, we recall the definition of $n$-representation infinite algebras and their higher preprojective algebras, as well as some results to prove or disprove the existence of higher preprojective gradings on some skew-group algebras. We give a detailed description of the interplay between Koszul-gradings and higher preprojective gradings, and state a crucial assumption on our gradings.

In \Cref{Sec: Abelian groups}, we summarise basics on the representation theory of abelian groups. 

In \Cref{Sec: Graded actions} we then explain a general strategy to classify $(n+1)$-preprojective algebras of type $\Tilde{A}$. This is used in \Cref{Sec: Classification1} to prove \Cref{Theo:SL3 Classification}. 

In \Cref{Sec: Tilings}, we reinterpret $3$-preprojective gradings as periodic lozenge tilings of the plane. We introduce invariants, classify them, and prove that gradings of the same invariant are related by mutation. This proves \Cref{Theo: Divisibility conditions}, \Cref{Theo: Mutation is transitive} and provides a new proof of \Cref{Theo:SL3 Classification}. 

\subsection*{Notation}
We use the following notation and conventions throughout. 
\begin{enumerate}
    \item Unless specified otherwise, we denote by $\Bbbk$ an algebraically closed field of characteristic $0$.
    \item All undecorated tensors $\otimes $ are taken over the field $\Bbbk$.
    \item We denote by $\diag(a_1, \ldots, a_d) $ the diagonal $(d \times d)$-matrix with diagonal entries $a_1, \ldots, a_d$. 
    \item We write $\frac{1}{n}(e_1, \ldots, e_d)$ for the diagonal matrix $\diag(\zeta^{e_1}, \ldots, \zeta^{e_d})$ of order $n$, where $\zeta$ is a primitive $n$-th root of unity which will be specified when necessary. 
    \item We denote by $n$ and $d$ positive integers, and by $G$ a finite group of order $n$ and by $R = \Bbbk [X_1 , \ldots, X_d]$ the polynomial ring in $d$ variables.
    \item A quiver $Q$ consists of vertices $Q_0$, arrows $Q_1$ and the source resp. target maps $s, t \colon Q_1 \to Q_0$.
    \item For $ a \in \mathbb{Z}$ and $n \in \mathbb{N}$, we write $a \bmod n$ for the smallest non-negative representative of the residue class $a + n\mathbb{Z}$. 
    \item For $ a, b \in \mathbb{Z}$ and $n \in \mathbb{N}$ we write $ a \equiv_n b$ if $a+ n \mathbb{Z} = b + n \mathbb{Z}$. We drop the subscript $n$ when no confusion is possible. 
\end{enumerate}

\section{Skew-group algebras and McKay quivers}\label{Sec:skew-group algebras}
Given a $\Bbbk$-algebra $A$ and a group $G$ acting on $A$ via automorphisms, one constructs the skew-group algebra as the tensor product 
\[ A \otimes  \Bbbk G \]
with multiplication given by 
\[ (a \otimes g) \cdot (b \otimes h) = (ag(b) \otimes gh).  \]
The particular setup that is of interest to us is the case where a finite subgroup $G \leq \SL_d(k)$ acts on a polynomial ring $R = \Bbbk[X_1, \ldots, X_d]$. More precisely, $G$ acts by left-multiplication on $V = \Bbbk^d$, and hence on the polynomial ring $R = \Bbbk[V]$. The dependence on $V$ will be suppressed when it is clear from context, in particular when we speak about a given subgroup $G \leq \SL_d(k)$. We will investigate properties of the skew-group algebra $R \ast G$. In the following, we collect known results which will be important for us.

\subsection{Superpotentials and McKay quivers}\label{SSec:SuperpotentialsandMcKay}
We recall some results from \cite{BSW}, and refer the reader to the article for details and proofs. In the article, Bocklandt, Schedler and Wemyss establish how to obtain a projective bimodule resolution of $R \ast G$ from the Koszul resolution of $R$. For this, we consider $N_i = R \otimes \bigwedge^i V \otimes R$, where $0 \leq i \leq d$. These are projective $R$-bimodules, and a projective bimodule resolution of $R$ is then given by
\[ 0 \rightarrow N_d \rightarrow N_{d-1} \rightarrow \cdots \rightarrow N_0 \rightarrow  R \rightarrow 0. \]
Note that the $N_i$ are already $(R \ast G)$-$R$-bimodules. We define $M_i = N_i \otimes \Bbbk G$ with $R\ast G$-bimodule structure given by 
\[ (r_1 \otimes g_1) (s_1 \otimes v \otimes s_2 \otimes h) (r_2 \otimes g_2) = r_1 g_1(s_1) \otimes g_1(v) \otimes s_2 h(r_2) \otimes hg_2    \] 
for $(r_1 \otimes g_1), (r_2 \otimes g_2) \in R \ast G$ and $(s_1 \otimes v \otimes s_2 \otimes g) \in R \otimes \bigwedge^i V \otimes R \otimes \Bbbk G = M_i$. 
From the sequence above, we obtain an exact sequence
\[ 0 \rightarrow M_d \rightarrow M_{d-1} \rightarrow \cdots \rightarrow M_0  \rightarrow R \otimes \Bbbk G \rightarrow 0. \]
By rewriting the terms in the sequence as 
\[ M_i = (R \otimes \Bbbk G )\otimes_{\Bbbk G} (\bigwedge^i V \otimes \Bbbk G) \otimes_{\Bbbk G} (R \otimes \Bbbk G) \]
we see that this is a projective $R \ast G$-bimodule resolution of $R \ast G$. Furthermore, this sequence can be used to establish that $R \ast G$ is $d$-Calabi-Yau (see \cite[Lemma 6.1]{BSW}). 
In fact, the authors prove that the skew-group algebra $R \ast G$ is a derivation quotient algebra of a superpotential. That means, there exists an element $(\omega \otimes 1) \in V^{\otimes d} \otimes \Bbbk G$ which is super-cyclically symmetric and which commutes with the $\Bbbk G$-action such that 
\[ R \ast G \simeq D(\omega \otimes 1, d-2). \]
Here, $D(\omega \otimes 1, d-e)$ is a \emph{derivation quotient algebra.} That means, it is the algebra obtained from the tensor algebra $T_{\Bbbk G}(V \otimes \Bbbk G)$ by factoring by the ideal generated by the $\Bbbk G$-bimodule 
\[ W_2 = \langle \delta_p (\omega) \otimes 1 \mid p \in V^{\otimes d-2} \rangle. \]
To define the derivation $\delta$, one fixes a $k$-basis of $V \otimes \Bbbk G$ and defines, for elementary tensors in the fixed $k$-basis $p \in V^{\otimes d-2}$, $v \in V^{\otimes d}$ and $w \in V^{\otimes 2}$, the map $\delta_p \colon V^{\otimes d} \to V^{\otimes 2}$ by 
\[ \delta_p(v) = \begin{cases} w &\textrm{if } v = p \otimes w \\ 0 &\textrm{otherwise} \end{cases}    \]
and extends bilinearly in $\rho$ and $v$. Furthermore, the term $M_i$ in the bimodule resolution can be seen to be isomorphic to the bimodule generated by the space of partial derivatives
\[ W_i = \langle \delta_p (\omega) \otimes 1 \mid p \in V^{\otimes d-i} \rangle.  \]
When passing to the Morita reduced version of $R \ast G$, the authors use this to explicitly construct a quiver $Q = (Q_0, Q_1)$ and relations of the reduced algebra. The quiver $Q$ is known to be the \emph{McKay quiver} of $(G, V)$. Precisely, the vertices of $Q$ are the irreducible $\Bbbk$-representations of $G$ up to isomorphism, and for two irreducibles $\chi_1, \chi_2 \in \operatorname{Irr}(G)$, the arrows $\chi_1 \rightarrow \chi_2$ form a basis of $\Hom_G(\chi_1 \otimes \rho, \chi_2)$, where $\rho$ is the representation afforded by $G$ acting naturally on $V$. The superpotential $\omega$ is then a $\Bbbk$-linear combination of cycles of length $d$ in $Q$, and the derivatives are taken with respect to paths of length $d-2$. We call the set of cycles of length $d$ which appear with a nonzero coefficient in $\omega$ the \emph{support} of $\omega$. We refer the reader to \cite[Section 3]{BSW} for details. 

\subsection{Cayley graphs}\label{SSec:Cayley graphs}
The skew-group algebra and the McKay quiver $Q$ have a special shape when the group $G \leq \SL_d(k)$ is abelian. In this case, the skew-group algebra is basic. Furthermore, any representation of $G$ decomposes into one-dimensional representations, and the set of isomorphism classes of one-dimensional representations, denoted $\hat{G} = \{ \chi_1, \ldots, \chi_n\}$, becomes a group under the tensor product. Fix the natural $d$-dimensional representation $\rho$ afforded by $G$ acting on $V$ and decompose it into one-dimensional representations $ \rho = \rho_1 \oplus  \ldots \oplus \rho_d$. Then there is an arrow $\chi_i \rightarrow \chi_j$ in $Q$ for every a summand $\rho_l$ in $\rho$ such that $\chi_i \otimes \rho_l \simeq \chi_j$. When the set $T = \{\rho_1, \ldots, \rho_d \}$ generates $\hat{G}$, we therefore obtain that $Q$ is the directed \emph{Cayley graph} (with multiplicities) of $\hat{G}$ with respect to the generating set (with multiplicities) $T$. The arrow corresponding to $\chi_i \otimes \rho_l \simeq \chi_j$ will be said to have \emph{type} $\rho_l$. Then, by \cite[Corollary 4.1]{BSW}, the support of the superpotential consist of all cycles of length $d$ which consist of arrows of $d$ distinct types.

\section{Higher preprojective algebras of higher representation infinite algebras}\label{Sec: HPAs}
Following Baer, Geigle and Lenzing in \cite{BGL}, the (classical) preprojective algebra $\Pi(\Lambda)$ of a finite-dimensional hereditary $\Bbbk$-algebra $\Lambda$ can be defined as a tensor algebra 
\[ \Pi(\Lambda) = T_\Lambda \Ext^1_\Lambda(D(\Lambda), \Lambda), \]
where, $D(\--) = \Hom_\Bbbk(\--, \Bbbk)$ is the $\Bbbk$-dual. One of the many appeals of this definition is that it readily generalises to Iyama's higher-dimensional Auslander-Reiten theory. 
\begin{Def}\cite[Definition 2.11]{IyamaOppermannStable}
Let $n$ be a positive integer, and $\Lambda$ be of global dimension at most $n$. The $(n+1)$-preprojective algebra of $\Lambda$ is 
\[ \Pi_{n+1}(\Lambda) = T_\Lambda \Ext^{n}_\Lambda(D(\Lambda), \Lambda). \]
\end{Def}
Note that as a tensor algebra, the $(n+1)$-preprojective algebra is graded, and that one recovers the original algebra as 
\[ \Lambda = \Pi_{n+1}(\Lambda)_0. \]
As in the classical case, the higher preprojective algebra encodes many properties of the algebra $\Lambda$. Of special interest to us is the property of being higher representation infinite. To give the definition by Herschend, Iyama and Oppermann in \cite{HIO}, we recall some notions from (higher) Auslander-Reiten theory, which can be found in the same article. Let $n$ be a positive integer and $\Lambda$ finite-dimensional $k$-algebra of global dimension at most $n$. We have the known \emph{Nakayama functor}
\[ \nu = D\RHom_\Lambda(\--,\Lambda) \colon \Db( \mod \Lambda) \to \Db( \mod \Lambda)  \] 
on the bounded derived category of finitely generated $\Lambda$-modules, with quasi-inverse
\[\nu^{-1} = \RHom_{\Lambda^{\mathrm{op}}}(D(\--), \Lambda) \colon \Db( \mod \Lambda) \to \Db( \mod \Lambda).\]
The \emph{derived higher AR-translation} is then defined as the autoequivalence 
\[\nu_n := \nu \circ [-n] \colon \Db(\mod \Lambda ) \to \Db( \mod \Lambda).\]

\begin{Def}\cite[Definition 2.7]{HIO} 
The algebra $\Lambda$ is called \emph{$n$-representation infinite} if for any projective module $P$ in $ \mod \Lambda$ and integer $i \geq 0$ we have 
\[ \nu_n^{-i} P \in   \mod \Lambda . \]
\end{Def}

Note that we see $\mod \Lambda$ as a subcategory of $\Db( \mod \Lambda)$ with objects concentrated in homological degree 0. The $n$-representation infinite algebras are characterised by certain homological properties of their correspoding higher preprojective algebras. In the following, $[n]$ denotes homological degree shift of a complex and $(-a)$ denotes grading shift of a complex of graded modules. 

\begin{Def}\cite[Definition 3.1]{AIR} \label{Def: nCY-GPa}
Let $\Gamma=\bigoplus_{i \geq 0}\Gamma_i$ be a positively graded $\Bbbk$-algebra. We call $\Gamma$ \emph{bimodule $n$-Calabi-Yau algebra of Gorenstein parameter $a$} if there exists a bounded graded projective $\Gamma$-bimodule resolution
$P_\bullet$ of $\Gamma$ and an isomorphism of complexes of graded $\Gamma$-bimodules
\[ P_\bullet \simeq \Hom_{\Gamma^{\mathrm{e}}}(P_\bullet, \Gamma^{\mathrm{e}})[n](-a).  \]
\end{Def}
The characterisation is then given as follows.

\begin{Theo}\cite[Theorem 4.35]{HIO}\label{Theo: HPG is f.d. GP1}
There is a bijection between the $n$-representation infinite algebras $\Lambda$ up to isomorphism and the bimodule $(n+1)$-Calabi-Yau algebras $\Gamma$ of Gorenstein parameter 1 with $\dim_\Bbbk \Gamma_i < \infty $ for all $i \in \mathbb{N}$ up to isomorphism. The bijection is given by 
\[ \Lambda \mapsto \Pi_{n+1}(\Lambda) \quad \mathrm{ and } \quad \Gamma \mapsto \Gamma_0.  \]
\end{Theo}

\begin{Rem}
The above theorem also appears in variants in \cite{KellerVandenBergh,MinamotoMori, AIR}. Note that it is crucial to assume that $\Gamma_0$ is finite-dimensional, since an $n$-representation infinite algebra is by definition finite-dimensional. An interesting question posed by Thibault in \cite[Question 4.3]{Thibault} is whether the definition, and more generally that of higher hereditary algebras, can be generalised to infinite-dimensional algebras.
\end{Rem}

Given the above theorem, we view the $(n+1)$-preprojective algebras of $n$-representation infinite algebras as bimodule $(n+1)$-Calabi-Yau algebras equipped with a special grading. Such a grading will therefore be called higher preprojective. 

\begin{Def}
Given a bimodule $(n+1)$-Calabi-Yau algebra $\Gamma$, a non-negative grading on $\Gamma$ such that $\Gamma$ is $(n+1)$-Calabi-Yau of Gorenstein parameter $1$ and such that $\dim_\Bbbk (\Gamma_i) < \infty$ for all $i$ is called a \emph{higher preprojective grading}.
\end{Def}

\begin{Rem}\label{Rem: Tension between dim and GP}
Note that the Calabi-Yau parameter $(n+1)$ and the Gorenstein parameter $a$ are in a certain sense independent. Given a graded algebra $\Gamma$ which is bimodule $(n+1)$-Calabi-Yau and of Gorenstein parameter $a$, the ungraded version is bimodule $(n+1)$-Calabi-Yau. Conversely, given an ungraded algebra which is bimodule $(n+1)$-Calabi-Yau, one may introduce gradings on the terms of the projective bimodule resolution such that the maps in the resolution respect these gradings, and then study the resulting Gorenstein parameter. The tension comes from requiring a small Gorenstein parameter and a finite-dimensional $\Gamma_0$. More precisely, if the Gorenstein parameter is $1$ and $\Gamma_0$ is infinite-dimensional, one may try to alleviate the situation by placing parts of $\Gamma_0$ in higher degrees, but that will increase the Gorenstein parameter, and vice versa. Indeed, for the skew-group algebras we consider, it is always possible to grade $\Gamma$ so that $\Gamma_0$ is finite-dimensional, or so that the Gorenstein parameter is $1$, but it is not always possible to achieve both, see \Cref{Ex:Klein-Group-Ungradeable}.   
\end{Rem}

The above discussion leads to the following natural question. 

\begin{Ques}
Given a bimodule $(n+1)$-Calabi-Yau algebra, can it be endowed with a higher preprojective grading?
\end{Ques}

In this article, we restrict our attention to the case of skew-group algebras $R \ast G$ of polynomial rings $R = \Bbbk [X_1, \ldots, X_d]$ by finite subgroups $G \leq \SL_d(k)$, which are known to be bimodule $d$-Calabi-Yau, as seen in \Cref{Sec:skew-group algebras}. 

\subsection{Koszulity and higher preprojective algebras}
An important feature of the skew-group algebras we consider is that they are Koszul with respect to the usual polynomial grading. Recall that a nonnegatively graded algebra $\Gamma = \bigoplus_{i \geq 0} \Gamma_i$ is called \emph{Koszul} if $\Gamma_0$ is semisimple and $\Gamma_0$ as a graded $\Gamma$-module admits a graded projective resolution where the $i$-th term in the resolution is generated in degree $i$. The interplay of Koszulity and higher preprojectivity has been investigated in \cite{GrantIyama}, and a Koszulity assumption is crucial for the quiver description of the higher preprojective algebra given in \cite[Theorem 3.6]{Thibault}. We fix a bimodule $d$-Calabi-Yau algebra $\Gamma$ and a grading such that $\Gamma = \bigoplus_{i \geq 0} \Gamma_i$ is Koszul, which we refer to as the \emph{Koszul grading}. 

\begin{Def}
    A higher preprojective grading realising $\Gamma = \Pi_{d}(\Lambda)$ is called a \emph{higher preprojective cut} (or \emph{cut} for short) if $\Gamma_1$ admits a decomposition $\Gamma_1 = (\Gamma_1)_0 \oplus (\Gamma_1)_1$ such that both $(\Gamma_1)_i$ are homogeneous of higher preprojective degree $i$ and $\Lambda = \langle \Gamma_0, (\Gamma_1)_0 \rangle$. 
\end{Def}

\begin{Rem}
    The terminology is justified when $\Gamma$ is basic. In this case, one can compute a quiver $Q$ and relations $I$ for $\Gamma \simeq \Bbbk Q/I$ such that $\Gamma$ can be graded by path-length and that the Koszul grading coincides with the path-length grading, i.e. one can choose as vertices idempotents in $\Gamma_0$ and as arrows elements respecting the decomposition $\Gamma_1 = (\Gamma_1)_0 \oplus (\Gamma_1)_1$. Then one obtains a quiver and relations for $\Lambda$ from $\Gamma$ by removing, or cutting, all arrows from $Q$ and all relations from $I$ which have positive preprojective degree. This coincides with the notion of a cut as defined in \cite{IyamaOppermann, HIO}.
\end{Rem}

Given a cut on a Koszul algebra, it is easy to produce higher preprojective gradings which are not cuts. In fact, we can even change the underlying semisimple base ring. Let us illustrate this in the classical case for hereditary representation infinite algebras. 

\begin{Exp}
Consider the algebra $\Gamma = \Bbbk [x,y] \ast C_2$ where $C_2 = \langle \frac{1}{2}(1,1) \rangle$ is cyclic of order $2$. This is the preprojective algebra of the hereditary representation infinite algebras of type $\Tilde{A}_1$. A quiver for $\Gamma$ is given by
\[ \begin{tikzcd}
e_1 \arrow[r, "a", bend left, shift left=4] \arrow[r, "b", bend left] & e_2 \arrow[l, "a^*"', bend left] \arrow[l, "b^*"', bend left, shift left=4]
\end{tikzcd}
\]
and the relations are generated by the commutator relations $\langle [a,a^\ast] + [b,b^\ast] \rangle$. The Koszul grading shall be given by $\Gamma_0 = \langle e_1, e_2 \rangle$ and $\Gamma_1 = \langle a,b,a^*, b^* \rangle$, so that it coincides with the path-length grading. We also have two preprojective gradings given by placing $a$, $b$ respectively $a^\ast$, $b^\ast$ in degree $1$. Both these gradings are clearly cuts.

Recall our convention that we compose paths so that $e_1 a \neq 0$. We now choose different idempotents
\begin{align*}
    \varepsilon_1 = e_1 + a \\
    \varepsilon_2 = e_2 - a
\end{align*}
and arrows 
\begin{alignat*}{2}
    \alpha &= \varepsilon_1 a \varepsilon_2 &&= a \\
    \beta &= \varepsilon_1 b \varepsilon_2 &&= b \\
    \alpha^\ast &= \varepsilon_2 a^* \varepsilon_1 &&= a^* + a^*a - aa^* - aa^*a \\
    \beta^* &= \varepsilon_2 b^* \varepsilon_1 &&= b^* + b^*a - ab^* - ab^*a.
\end{alignat*}
for the same algebra. The quiver we obtain is   
\[ \begin{tikzcd}
\varepsilon_1 \arrow[r, "\alpha", bend left, shift left=4] \arrow[r, "\beta", bend left] & \varepsilon_2 \arrow[l, "\alpha^*"', bend left] \arrow[l, "\beta^*"', bend left, shift left=4]
\end{tikzcd}
.\]
It is easy to check that the new arrows satisfy the corresponding commutator relations. 

We consider two preprojective gradings given by ``cutting'' the new quiver, i.e. placing $\alpha^*, \beta^*$ respectively $\alpha, \beta$ in degree $1$. Thus, we obtain hereditary representation infinite pieces
\begin{align*}
    \Lambda^1 &= \langle \varepsilon_1, \varepsilon_2, \alpha, \beta \rangle, \\
    \Lambda^2 &= \langle \varepsilon_1, \varepsilon_2, \alpha^*, \beta^* \rangle.
\end{align*}
These two algebras interact quite differently with the chosen Koszul grading. In fact, the algebra $\Lambda^1$ corresponds to the first cut we considered since 
\[ \Lambda^1 = \langle e_1, e_2, a,b \rangle. \]
However, $\Lambda^2$ is not a Koszul graded subalgebra of $\Gamma$. Furthermore, we even have $e_1, e_2 \not \in \Lambda^2$. 

In the spirit of exploration, let us consider the element $a$ which we added to the idempotents. Since $a$ is nilpotent, we have that $r = (1+a)$ must be invertible. Indeed, we have $(1 + a) (1-a) = 1 + a - a - a^2 = 1$. Therefore, we can conjugate all of $\Gamma$ with $r$. We have 
\begin{alignat*}{3}
    e_1^r &= (1-a) e_1 (1+a) &&= e_1 + a &&= \varepsilon_1, \\
    e_2^r &= (1-a) e_2 (1+a) &&= e_2 - a &&= \varepsilon_2, \\
    a^r &= (1-a) a (1+a) &&= a &&= \alpha, \\
    b^r &= (1-a) b (1+a) &&= b &&= \beta, \\
    (a^*)^r &= (1-a)a^*(1+a) &&= a^* - aa^* + a^* a - aa^* a &&= \alpha^*, \\
    (b^*)^r &=(1-a)b^*(1+a) &&= b^* - ab^* + b^* a - ab^* a &&= \beta^*. 
\end{alignat*}
This means that we found an automorphism of $\Gamma$ which takes $\Lambda^i$ to the corresponding cuts we computed from the original quiver. Alternatively, one can interpret the automorphism as producing a different Koszul grading for which the algebras $\Lambda^i$ correspond to cuts. 
\end{Exp}

Let us make some remarks based on the above example. 
\begin{Rem} Fix a Koszul grading $\Gamma = \bigoplus_{i \geq 0} \Gamma_i$, and let $\Lambda$ be the degree $0$ piece of a higher preprojective grading on $\Gamma$. Then we note the following.
    \begin{enumerate}
        \item An arbitrary higher preprojective grading on $\Gamma$ need not be a cut. 
        \item The algebra $\Lambda$ need not contain $\Gamma_0$.  
    \end{enumerate}
\end{Rem}

While $\Lambda$ need not arise from a cut for an arbitrary Koszul grading on $\Pi_d(\Lambda)$, we are not aware of an example where $\Lambda$ is not \emph{conjugate} to a subalgebra arising from a cut. We therefore pose the following question. 

\begin{Ques}\label{Ques:Compatibility of Koszul and HPPG}
    Let $\Lambda$ be $(d-1)$-representation infinite and $\Gamma = \Pi_d(\Lambda)$. Suppose we are given a Koszul grading on $\Gamma$. Is $\Lambda$ conjugate to a subalgebra of $\Gamma$ arising from a cut with respect to the Koszul grading? 
\end{Ques}

A positive answer to the above question would complement the results of \cite{GrantIyama}, showing that $\Lambda$ has a Koszul grading if and only if $\Pi_d(\Lambda)$ has a Koszul grading, and that these are unique up to conjugation. Furthermore, in light of \cite{Thibault}, this would show that in order to disprove the existence of a higher preprojective grading, one may fix a Koszul grading and work directly with the corresponding quiver. 

From now on, we will mainly consider higher preprojective gradings which are cuts. If \Cref{Ques:Compatibility of Koszul and HPPG} can be answered positively, this amounts simply to performing a change of basis. We then have the advantage that we have a direct implication for the superpotential, which is proven and used in \cite[Theorem 4.11]{Thibault}.

\begin{Pro}\label{Pro: Potential is homogeneous}
    Write $\Gamma = R \ast G = D(\omega, d-2) = T_{\Bbbk G}(V \otimes \Bbbk G)/\langle W_2 \rangle$ as in \Cref{Sec:skew-group algebras}. Then a higher preprojective cut $V \otimes \Bbbk G = (V \otimes \Bbbk G)_0 \oplus (V \otimes \Bbbk G)_1$ on $R \ast G$ induces a grading on $T_{\Bbbk G}(V \otimes \Bbbk G)$, and $\omega$ is homogeneous of degree $1$ with respect to this grading. 
\end{Pro}

\begin{proof}
    By \cite[Lemma 3.8]{AIR}, the bimodule $\Gamma \otimes W_0 \otimes \Gamma$ generated by the $0$-th derivatives of $\omega$, i.e.\ by $\omega$, is generated in preprojective degree $1$. By \cite[Lemma 4.9]{Thibault}, this grading coincides with the induced grading on $T_{\Bbbk G}((V \otimes \Bbbk G)_0 \oplus (V \otimes \Bbbk G)_1)$, hence $\omega$ is homogeneous of degree $1$.  
\end{proof}

Let us illustrate how this leads to a visual computation of possible higher preprojective cuts. 

\begin{Exp}\label{Ex:Klein-Group-Ungradeable}
Consider $G  = \langle \frac{1}{2}(1,1,0), \frac{1}{2}(1,0,1) \rangle \leq \SL_3(k)$. Since $G \simeq C_2 \times C_2$ is abelian, we can view the McKay quiver as a Cayley graph, as detailed in \Cref{SSec:Cayley graphs}. The group has four elements, so the quiver has four vertices. We label the vertices by $\{ (0,0), (0,1), (1,0), (1,1) \}$. The representation in $\SL_3(k)$ given by the inclusion of $G$ decomposes as the direct sum of the three non-trivial irreducible representations of $G$. Therefore every vertex of the quiver has an arrow to every other vertex. We draw the quiver below, and so that the horizontal, vertical, and diagonal arrows are of the same type respectively.   
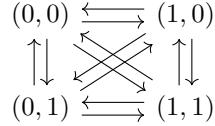
\[\begin{tikzcd}
	(0,0) & (1,0) \\
	(0,1) & (1,1)
	\arrow[shift right=1, from=1-1, to=1-2]
	\arrow[shift right=1, from=1-2, to=1-1]
	\arrow[shift left=1, from=1-2, to=2-2]
	\arrow[shift left=1, from=2-2, to=1-2]
	\arrow[shift right=1, from=2-1, to=2-2]
	\arrow[shift right=1, from=2-2, to=2-1]
	\arrow[shift left=1, from=1-1, to=2-1]
	\arrow[shift left=1, from=2-1, to=1-1]
	\arrow[shift right=1, from=2-1, to=1-2]
	\arrow[shift right=1, from=1-2, to=2-1]
	\arrow[shift right=1, from=1-1, to=2-2]
	\arrow[shift right=1, from=2-2, to=1-1]
\end{tikzcd}\]
Next, we recall that the superpotential $\omega$ is given by a linear combination of all $3$-cycles consisting of arrows of three different types. Let us grade the quiver, trying to find a higher preprojective cut. Placing an arrow in degree $1$ corresponds to removing (or cutting) it from $(R \ast G)_0$. By symmetry, we may begin by removing the horizontal arrow $(0,0) \to (1,0)$. Since the superpotential needs to be homogeneous of degree $1$ in order to obtain a Gorenstein parameter $1$, the paths of length two which can be completed with this arrow into cycles in the support of $\omega$ need to remain in $(R\ast G)_0$. We draw those blue in the picture below. 
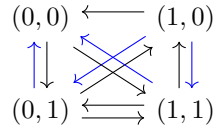
\[\begin{tikzcd}
	(0,0) & (1,0) \\
	(0,1) & (1,1)
	\arrow[shift right=1, from=1-2, to=1-1]
	\arrow[shift left=1, from=1-2, to=2-2, blue]
	\arrow[shift left=1, from=2-2, to=1-2]
	\arrow[shift right=1, from=2-1, to=2-2]
	\arrow[shift right=1, from=2-2, to=2-1]
	\arrow[shift left=1, from=1-1, to=2-1]
	\arrow[shift left=1, from=2-1, to=1-1, blue]
	\arrow[shift right=1, from=2-1, to=1-2]
	\arrow[shift right=1, from=1-2, to=2-1, blue]
	\arrow[shift right=1, from=1-1, to=2-2]
	\arrow[shift right=1, from=2-2, to=1-1, blue]
\end{tikzcd}\]
Now, consider the two vertical arrows $(1,0) \rightleftarrows (1,1)$. If both remain in $(R \ast G)_0$, it is easy to see that they will generate an infinite-dimensional subalgebra in $(R \ast G)_0$. Hence, one of them needs to be placed in degree $1$, but the only choice we have is the arrow $(1,1) \to (1,0)$. Removing it and colouring the arrows which then must remain for the Gorenstein parameter not to exceed $1$, we obtain the following. 
\[\begin{tikzcd}
	(0,0) & (1,0) \\
	(0,1) & (1,1)
	\arrow[shift right=1, from=1-2, to=1-1, blue]
	\arrow[shift left=1, from=1-2, to=2-2, blue]
	\arrow[shift right=1, from=2-1, to=2-2, blue]
	\arrow[shift right=1, from=2-2, to=2-1]
	\arrow[shift left=1, from=1-1, to=2-1]
	\arrow[shift left=1, from=2-1, to=1-1, blue]
	\arrow[shift right=1, from=2-1, to=1-2]
	\arrow[shift right=1, from=1-2, to=2-1, blue]
	\arrow[shift right=1, from=1-1, to=2-2, blue]
	\arrow[shift right=1, from=2-2, to=1-1, blue]
\end{tikzcd}\]
But now it is easy to see that the two blue arrows $(0,0) \rightleftarrows (1,1)$ generate an infinite-dimensional subalgebra in $(R \ast G)_0$, but both need to remain in order for the Gorenstein parameter not to exceed $1$. Note that our choice of the first arrow to place in degree $1$ was, by symmetry, arbitrary, so we can conclude that $R \ast G$ can not be endowed with a higher preprojective cut. 
However, in the spirit of exploration, and to illustrate \Cref{Rem: Tension between dim and GP}, let us construct some different gradings. Recall that $R = \Bbbk[X_1, X_2, X_3]$, so we can simply place the variable $X_1 \otimes 1$ in degree $1$ to obtain a grading on $(R \ast G)$ that has Gorenstein parameter $1$. This corresponds to removing all diagonal arrows, and leaves us with an infinite-dimensional algebra $(R \ast G)_0$ with quiver 
\[\begin{tikzcd}
	(0,0) & (1,0) \\
	(0,1) & (1,1)
	\arrow[shift right=1, from=1-1, to=1-2]
	\arrow[shift right=1, from=1-2, to=1-1]
	\arrow[shift left=1, from=1-2, to=2-2]
	\arrow[shift left=1, from=2-2, to=1-2]
	\arrow[shift right=1, from=2-1, to=2-2]
	\arrow[shift right=1, from=2-2, to=2-1]
	\arrow[shift left=1, from=1-1, to=2-1]
	\arrow[shift left=1, from=2-1, to=1-1]
\end{tikzcd}\]
Conversely, by simply placing  all arrows in degree $1$, we obtain a grading with $(R \ast G)_0 \simeq \Bbbk G$, which is finite-dimensional, but now $\omega$ is homogeneous of degree $3$, hence $(R \ast G)$ with this grading has Gorenstein parameter $3$. 
\end{Exp}

\subsection{Amiot-Iyama-Reiten gradings}
Assume that $G \leq \SL_d(k)$ is a cyclic group of order $n$, and that it has a generator $g = \frac{1}{n}(e_1 , \ldots, e_d)$ with $\sum_{i=1}^d e_i = n$ and $e_i > 0$ for all $i$. In \cite{AIR}, Amiot, Iyama and Reiten have given a construction of cuts of the skew-group algebra $R \ast G$. Their construction works as follows: Since $G$ is abelian, we have that $R \ast G \simeq \Bbbk Q/I$, where $Q$ is the McKay quiver of $G$. Since $G$ is cyclic, we can label the vertices of $Q$ by $\{ 0, \ldots , (n-1)\}$ so that every arrow is of the form $l \rightarrow (l + e_i) \bmod n $. Then, grade the arrows via 
\[ \deg( l \rightarrow (l+e_i) \bmod n) = \begin{cases} 1, & l > (l + e_i) \bmod n, \\ 0 , & \text{else.} \end{cases}  \]

\begin{Theo}\cite[Theorem 5.6]{AIR} \label{Theo:AIR cuts} 
With the above grading, the skew-group algebra $R \ast G$ is a $d$-preprojective algebra. 
\end{Theo} 

\begin{Def}
Under the same assumptions, we call a grading arising by the above construction an \emph{Amiot-Iyama-Reiten-cut} or \emph{AIR-cut} for short. 
\end{Def}

\begin{Rem}\label{Rem:SL3-Cyclic-AIR possible}
One can see from the construction that the quiver of the degree 0 part is acyclic. Furthermore, we can see that the condition $\sum_{i=1}^d e_i = n$ can always be fulfilled for some generator $g = \frac{1}{n}(e_1 , \ldots, e_d)$ if $d=3$. Indeed, the sum has to be $n$ or $2n$, since we assume that $\det(g) = 1$. If the sum is $2n$, one can replace the generator $g$ with its inverse. The inverse is of the form $\frac{1}{n}(n - e_1 , \ldots, n - e_d)$, so for $d=3$ the exponents sum to $\sum_{i=1}^3 (n- e_i) = 3n - \sum_{i=1}^3 e_i = n$. 
\end{Rem} 

\begin{Exp}\label{Ex: C3-AIR}
    Consider the group $G = \langle \frac{1}{3} (1,1,1) \rangle $ of order $3$. The McKay quiver has $3$ vertices, which we label by $\{ 0,1,2\}$. The arrows then all take vertex $i$ to $i+1 \bmod 3$. Hence the quiver is given by 
    
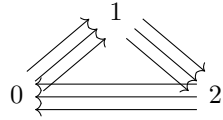
\[\begin{tikzcd}
	& 1 \\
	0 && 2
	\arrow[shift left=2, from=2-1, to=1-2]
	\arrow[shift right=2, from=2-1, to=1-2]
	\arrow[from=2-1, to=1-2]
	\arrow[shift left=2, from=1-2, to=2-3]
	\arrow[shift right=2, from=1-2, to=2-3]
	\arrow[from=1-2, to=2-3]
	\arrow[shift left=2, from=2-3, to=2-1]
	\arrow[shift right=2, from=2-3, to=2-1]
	\arrow[from=2-3, to=2-1]
\end{tikzcd}\]
An AIR-cut then is given by placing the arrows $2 \to 0$ in degree $1$. Indeed, one can easily check that the remaining degree $0$ part has quiver 
\[\begin{tikzcd}
	0 & 1 & 2
	\arrow[shift left=2, from=1-1, to=1-2]
	\arrow[shift right=2, from=1-1, to=1-2]
	\arrow[from=1-1, to=1-2]
	\arrow[shift left=2, from=1-2, to=1-3]
	\arrow[shift right=2, from=1-2, to=1-3]
	\arrow[from=1-2, to=1-3]
\end{tikzcd}\]
and relations making it into a Beilinson-algebra $B(2)$, which is known to be $2$-representation infinite by \cite[Example 2.15]{HIO}.
\end{Exp}

\subsection{Thibault's criterion}
Assume that $G \leq \SL_d(k)$ is a cyclic group of order $n$, and that it has a generator $g = \frac{1}{n}(e_1 , \ldots, e_d)$. One can ask whether the condition $\sum_{i=1}^d e_i = n$ on the generator is necessary to find a higher preprojective grading. Dropping this assumption lead Thibault to the discovery of the following negative criterion.

\begin{Theo}\cite[Theorem 5.15]{Thibault} \label{Theo:TCrit}
Suppose that the embedding of a finite subgroup $G \leq \SL_d(k)$ factors through some embedding of $\SL_{d_1}(k) \times \SL_{d_2}(k)$ into $\SL_d(k)$ for $d_1 + d_2 = d$. Then $R \ast G$ can not be endowed with the grading of a $d$-preprojective algebra.
\end{Theo}

Note that the criterion works for an arbitrary finite group $G$. The proof shows that any grading giving $R \ast G$ Gorenstein parameter $1$ leaves an infinite-dimensional degree $0$ part. 

\subsection{Skewed higher preprojective gradings}
If we have a given higher preprojective grading on an algebra, one can ask how it behaves under skewing by a group action. The following theorem by Le Meur \cite{LeMeur} gives an answer to this question under the assumption that the action is by graded automorphisms. 

\begin{Theo}\cite[Proposition 6.2.1]{LeMeur} \label{Theo: LeMeur}
Suppose $\Lambda$ is a finite-dimensional $\Bbbk$-algebra with $\gdim(\Lambda) \leq d-1$, and suppose that the finite group $G$ acts on $\Lambda$ by automorphisms. Further, suppose that $\chr \Bbbk \nmid G$. Then we have the following:
\begin{enumerate}
    \item $\Lambda$ is $(d-1)$-representation infinite if and only if $\Lambda \ast G$ is $(d-1)$-representation infinite.
    \item If $\Lambda$ is $(d-1)$-representation infinite, then $G$ acts on the $d$-preprojective algebra $\Pi_d(\Lambda)$, and the skewed preprojective algebra $\Pi_d(\Lambda) \ast G$ is isomorphic to the $d$-preprojective algebra $\Pi_d(\Lambda \ast G)$. 
\end{enumerate}
\end{Theo}

\begin{Def}
Under the same assumptions, we call the $d$-preprojective grading on $\Pi_d(\Lambda) \ast G \simeq \Pi_d(\Lambda \ast G) $ the \emph{skewed grading} of the $d$-preprojective grading on $\Pi_d(\Lambda)$.
\end{Def}

\section{Abelian subgroups of \texorpdfstring{$\SL_d(k)$}{SL(k)}}\label{Sec: Abelian groups}
In this section, we fix an abelian group $G$, and assume that it embeds in $\SL_d(k)$. We describe the possible forms of $G$ up to isomorphism, and once an isomorphism type of $G$ is fixed we describe all possible ways $G$ can arise as a subgroup up to conjugation, i.e.\ we describe all relevant representations of $G$. These statements are well-known, but we prove them here for convenience.

To begin with, we decompose the finite abelian group $G$ into invariant factors. That means, we fix a minimal set of positive integers $m_1, m_2, \ldots, m_l$ such that $m_i \mid m_{i+1}$ for $1 \leq i < l$ and 
\[ G \simeq C_{m_1} \times \cdots \times C_{m_l}. \]

Next, we characterise the faithful representations of an abelian group $G$ in terms of the dual group $\hat{G} = \Hom(G, \Bbbk^\ast)$. Note that for an abelian group, we have $G \simeq \hat{G}$. For a subgroup $H \leq G$ one defines \[H^\bot = \{ \chi \in \hat{G} \mid H \subseteq \Ker(\chi) \} \leq \hat{G},\] and similarly for $U \leq \hat{G}$ one defines \[U^\bot = \{ g \in G \mid \chi(g) = 1 \textrm{ for all } \chi \in U \} \leq G.\] The following fact from character theory is well known, see \cite[Problem 2.7]{isaacsCT}.

\begin{Lem}
Let $G$ be an abelian group and $H \leq G$. The assignment $H \mapsto H^\bot$ is a lattice antiisomorphism between the subgroup lattices of $G$ and of $\hat{G}$. The inverse is given by $U \mapsto U^\bot$ for $U \leq \hat{G}$.
\end{Lem}

\begin{Lem}\label{Lem:Faithful_Iff_Generating}
Given a representation $\rho$ of an abelian group $G$, decompose it into its irreducible direct summands $\rho = \chi_1 \oplus \cdots \oplus \chi_n$. Then $\rho$ is faithful if and only if $\{ \chi_1 , \ldots , \chi_n\}$ is a generating set of the dual group $\hat{G}$. 
\end{Lem}

\begin{proof}
Let $U \leq \hat{G}$ be arbitrary. After choosing a generating set $\{u_1, \ldots, u_l  \}$ of $U$, we can rewrite the above correspondence as $ U^\bot = \bigcap_{u \in U} \Ker(u) = \bigcap_{i = 1}^l \Ker(u_i)$.
Under this correspondence, we thus obtain that 
\[ \Ker(\rho)^\bot = (\bigcap_{i=1}^n \Ker(\chi_i))^\bot = \langle \chi_1, \ldots, \chi_n \rangle, \] so $\Ker(\rho) = \{ 1\}$ if and only if $\langle \chi_1, \ldots, \chi_n \rangle = \hat{G}$. 
\end{proof}

In light of the previous lemma, we need a bound on the number of generators of $G$, which will also apply to $\hat{G}$ since the groups are isomorphic.

\begin{Lem}\label{Lem:Number_Of_Generators_Of_Abelian}
Every generating set of $G \simeq C_{m_1} \times \cdots \times C_{m_l}$ has at least $l$ elements. 
\end{Lem}

\begin{proof}
We use that $G$ has an elementary abelian quotient, obtained by picking a prime $p \mid m_1$, considering the subgroup $H = \{ g^p \mid g \in G \}$ and then taking the quotient $G/H \simeq (C_p)^l$. Then, since $(C_p)^l$ is a vector space over $\mathbb{F}_p$, every generating set of $(C_p)^l$ has at least $l$ elements, so the same is true for $G$. 
\end{proof}

We can now give a tight bound on the number of invariant factors of $G \leq \SL_d(k)$.

\begin{Lem}
Suppose $G \simeq C_{m_1} \times \cdots \times C_{m_l}$ embeds into $\GL_{d}(k)$ for some $d \in \mathbb{N}$. Then the number of invariant factors $l$ satisfies $l \leq d$.
\end{Lem}

\begin{proof}
We denote by $\rho$ the representation of $G$ that is given by the inclusion in $\GL_d(k)$, which means that $\rho$ is a faithful representation of $G$. Combining \Cref{Lem:Faithful_Iff_Generating} with \Cref{Lem:Number_Of_Generators_Of_Abelian} shows that $\rho$ must have at least $l$ irreducible direct summands. Since each summand has dimension 1, we have that $\rho$ has at least dimension $l$, which proves the claim.
\end{proof}

\begin{Lem}\label{Lem:Number_Of_Cyclic_Factors}
Suppose $G \simeq C_{m_1} \times \cdots \times C_{m_l}$ embeds into $\SL_d(k)$ for some $d \in \mathbb{N}$. Then the number of invariant factors $l$ satisfies $l \leq d-1$.
\end{Lem}

\begin{proof}
We denote by $\rho$ the representation of $G$ that is given by the inclusion in $\SL_d(k)$. Since $G$ is abelian, we can decompose $\rho = \chi \oplus \rho' $ into a representation $\chi$ of dimension 1 and a representation $\rho'$ of dimension $d-1$. Furthermore, the condition $\rho(G) \leq \SL_d(k)$ means that we have $\chi = \det(\rho')^{-1}$. We note that $\rho'$ is a faithful representation of $G$, since any element $g \in \Ker(\rho')$ certainly satisfies $\det(\rho'(g)) = 1$, so $\Ker(\rho') \subseteq \Ker(\rho) = \{1\}$. 
Thus, we have found an embedding $ \rho' \colon G \to \GL_{d-1}(k)$ and the previous lemma gives our claim.  
\end{proof}

We apply the statements of this section to our situation of interest and consider abelian subgroups $G$ of $\SL_3(k)$. By \Cref{Lem:Number_Of_Cyclic_Factors}, these groups have at most $2$ invariant factors, so they are either cyclic groups $C_m$ or products of cyclic groups $C_{m_1} \times C_{m_2}$ where $m_1 \mid m_2$. By \Cref{Lem:Faithful_Iff_Generating}, any faithful representation $\rho = \chi_1 \oplus \chi_2 \oplus \chi_3$ in $\SL_3(k)$ of the group must give a generating set $\{ \chi_1, \chi_2, \chi_3\}$ of $\hat{G}$. But since the representation is in $\SL_3(k)$, we know that $\chi_3 = (\chi_1 \cdot \chi_2)^{-1}$, so in fact we have that $\{ \chi_1, \chi_2 \}$ must be a generating set of $\hat{G}$. Thus, for the non-cyclic case, we must indeed consider all minimal generating sets of $G$. It will be important for us to find certain direct product decompositions of $G$ that are compatible with a chosen generating set of $\hat{G}$. 

\begin{Lem}\label{Lem:Factor_Decomp}
Suppose the embedding $G = C_m \times C_{lm} \leq \SL_3(k)$ is given by $\rho = \chi_1 \oplus \chi_2 \oplus \chi_3$. Further, suppose that no summand $\chi_i$ is trivial, and that $|G | > 4$. Then there exists a direct factor decomposition $G = H \times K$ where $H$ is cyclic and such that none of the restrictions $\chi_i \restriction_H$ are trivial.
\end{Lem}

\begin{proof}
We work in the dual group $\hat{G}$. Note that for $H \leq G$ we have a natural isomorphism $ \hat{G}/H^\bot \to \hat{H}, (\chi \cdot  H^\bot) \mapsto \chi \restriction_H$. Thus, restricting the representation $\chi_i$ to a subgroup $H \leq G$ corresponds to taking the coset of $\chi_i$ in $\hat{G}/H^\bot$. It therefore suffices to find a direct product decomposition $ \hat{G} = U \times V $ such that none of the $\chi_i$ lies in $U$ and $V$ is cyclic. Then we realise the claimed decomposition by taking $H = U^\bot$ and $K = V^\bot$, which yield that the $\chi_i \restriction_H$ are not trivial and that $V \simeq \hat{G}/U = \hat{G}/(H^\bot) \simeq \hat{H} \simeq H$ is cyclic. Since $\hat{G} \simeq G$, we pick $x,y \in \hat{G}$ such $|x| = m$ and $|y|=lm$, i.e. we have $ \hat{G} = \langle x \rangle \times \langle y \rangle = C_m \times C_{lm}$. 

If $\chi_1, \chi_2, \chi_3 \not \in C_m \times \{1\}$ we are done. 
Otherwise, we assume without loss of generality that $\chi_1 \in C_m \times \{1\}$. Then we obtain that $| \chi_1|= m$, and hence $| \chi_2| = lm = | \chi_3|$. Note that $S = \langle (x, y^l) \rangle$ has order $m$ and that $C_m \times \{1\} \neq S$, so $\chi_1$ can not be in $S$. Furthermore, $S$ intersects $ \{1 \} \times C_{lm}$ trivially, so we have a direct product decomposition of $C_m \times C_{lm}$ into $S$ and $\{ 1 \} \times C_{lm}$ with $\chi_1 \not \in S$. 

If $\chi_2, \chi_3 \not \in S$ we are done. Otherwise, we assume that $\chi_2 \in S$. Recall that $S$ has order $m$, so we must have $l=1$. Note that then the subgroup $\{ 1 \} \times C_{lm}$ has order $m$ and that it differs from both $S$ and $C_m \times \{1\}$. Therefore, we have $\chi_1, \chi_2 \not \in \{ 1 \} \times C_{lm}$. 

If $\chi_3 \not \in \{ 1 \} \times C_{lm}$, we are done.  Otherwise, we assume that $\chi_3 \in \{ 1 \} \times C_{lm}$. Then, we consider $T = \langle (x,y^{-l}) \rangle$, which has a direct complement $\{1 \} \times C_{lm}$. By construction it is clear that $\chi_1, \chi_3 \not \in T$. If $\chi_2 \in T$, we have $S = T$, so $m=2$ and therefore $|G| = m^2 l = 4$. But we assumed $|G| >4$, so we must have $\chi_2 \not \in T$ and we are done.  
\end{proof}

\section{Graded actions of direct factors}\label{Sec: Graded actions}
The purpose of this section is to show that \Cref{Theo: LeMeur} can be applied in the following way. Given a group $G = H \times K \leq \SL_d(k)$, note that 
\[ R \ast G = (R \ast H) \ast K,  \] 
where $K$ acts on $(R \ast H)$ via $g \cdot (r \otimes h) = g(r) \otimes h$ for $g \in K$. Suppose that $R \ast H$ has a higher preprojective cut. In particular, that means that $R \ast H = \Pi_d(\Lambda)$ is the $d$-preprojective algebra of the $(d-1)$-representation infinite algebra $\Lambda = (R \ast H)_0$. Then we would like to show that $K$ acts by graded automorphisms on $R \ast H$, since then $K$ acts on $\Lambda$ via automorphisms. Thus, by \Cref{Theo: LeMeur}, we then have that $\Lambda \ast K$ is also $(d-1)$-representation infinite, and that the corresponding higher preprojective algebra is 
\[ \Pi_d(\Lambda \ast K) = \Pi_d(\Lambda) \ast K=  (R \ast H) \ast K = R \ast G .  \]
This will hold true when all arrow spaces of $R \ast H$ are homogeneous, which is the case when $H$ is abelian. Recall that in this case, $R \ast H$ is basic.  

\begin{Lem}\label{Lem: Double arrows have same degree}
Let $G \leq \SL_d(k)$ be an abelian group and fix a higher preprojective cut on $R \ast G \simeq \Bbbk Q/I $, where the quiver is chosen so that every arrow is homogeneous for the preprojective grading. Then any two parallel arrows have the same degree, i.e. given $\alpha, \beta \in Q_1$ with $s(\alpha) = s(\beta)$ and $t(\alpha) = t(\beta)$ we have $\deg(\alpha) = \deg(\beta)$.  
\end{Lem}

\begin{proof}
By \Cref{Pro: Potential is homogeneous}, the superpotential $\omega$ on $R \ast G$ is homogeneous of preprojective degree $1$. Recall from \Cref{SSec:Cayley graphs} that since $G$ is abelian, the superpotential $\omega$ is a linear combination of all cycles of length $d$ consisting of arrows of $d$ distinct types. Suppose now that $\deg(\alpha) = 0$ and $\deg(\beta)=1$. We know that $\beta$ can be completed into a cycle $\beta \alpha_2 \gamma_3 \cdots \gamma_d$ which is in $\supp(\omega)$, and we can choose it so that $\alpha_2$ and $\alpha$ have the same type. Since $\deg(\beta) = 1$, we know that $\deg(\gamma_i) = 0$ for $3 \leq i \leq d$ and $\deg(\alpha_2) = 0$. Furthermore, there is an arrow $\beta_2$ parallel to $\alpha_2$ and which has the same type as $\beta$. Hence, the cycle $\alpha \beta_2 \gamma_3 \cdots \gamma_d$ is also in $\supp(\omega)$. Since all other arrows in this cycle have degree $0$, we see that $\deg(\beta_2)=1$. Since $Q$ is finite, we can repeat the argument until we find that there is a cycle $\alpha \alpha_2 \cdots \alpha_N$ for some $N > 0$ which is homogeneous of degree $0$. Since all involved arrows $\alpha_i$ have the same type, this cycle is not nilpotent, hence it generates an infinite-dimensional subalgebra of $(R \ast G)_0$. Hence, the grading is not higher preprojective. 
\end{proof}

\begin{Pro}\label{Prop:Lifting by abelian}
Let $G = H \times K \leq \SL_d(k)$ and fix a higher preprojective cut on $R \ast H$. Then $K$ acts on $R \ast H$ via $g \cdot (r \otimes h) = g(r) \otimes h$ for $g \in K$. If $H$ is abelian, this action is graded. 
\end{Pro}

\begin{proof}
We first note that $K$ fixes the idempotents of $R \ast H$, as these arise precisely as $1 \otimes e$ for some idempotent $e \in \Bbbk H$. Indeed, it is clear that $g \cdot (1 \otimes e) = (g(1) \otimes e) = (1 \otimes e)$ for any $g \in K$. Then, consider the spaces of homogeneous arrows $V_i = (V \otimes \Bbbk H)_i$ for $i=0,1$. Further, fix homogeneous primitive idempotents $(1\otimes e),(1 \otimes f) \in R \ast  H$. By \Cref{Lem: Double arrows have same degree}, one of the spaces $e(V_0)f$ and $e(V_1)f$ is zero, so we may write $e(V)f = e(V_j)f$ for some fixed $j$. Since $K$ commutes with idempotents, we have $ K(e(V_j)f) = K(e(V)f) = e(K(V))f  = e(V)f = e(V_j)f $. Thus, $K$ preserves the space of homogeneous elements $e(V_j)f$, and since $e$ and $f$ were arbitrary, $K$ preserves $V_0$ and $V_1$, which means the action is graded.  
\end{proof}

\begin{Cor} \label{Cor: Lifting strategy}
    Let $G = H \times K \leq \SL_d(k)$ such that $H$ is abelian. If there exists a higher preprojective cut on $R \ast H$, then there exists a higher preprojective cut on $R \ast G$. 
\end{Cor}

\begin{proof}
    Apply \Cref{Prop:Lifting by abelian} and \Cref{Theo: LeMeur} to $R \ast G \simeq (R \ast H) \ast K$. 
\end{proof}

\section{Classifying 3-preprojective structures of type \texorpdfstring{$\Tilde{A}$}{A~}} \label{Sec: Classification1}
We are now ready to give the complete classification of $3$-preprojective algebras of type $\Tilde{A}$. Recall that type $\Tilde{A}$ was defined in \cite{HIO} to be exactly those $(d-1)$-representation infinite algebras whose higher preprojective algebras are of the form $R \ast G$ for some abelian group $G$. 

\begin{Theo}\label{Theo:SL3 Classification}
Let $G \leq \SL_3(k)$ be abelian, acting on $R = \Bbbk[X_1, X_2, X_3]$ via change of variables. Then $R \ast G$ can be endowed with a higher preprojective grading if and only if $G$ is not isomorphic to $C_2 \times C_2$ and the embedding of $G$ into $\SL_3(k)$ does not factor through any embedding of $\SL_1(k) \times \SL_2(k)$ into $\SL_3(k)$.  
\end{Theo}

\begin{proof}
If the embedding of $G$ into $\SL_3(k)$ factors through some embedding of $\SL_1(k) \times \SL_2(k)$ into $\SL_3(k)$, we know by Thibault's criterion \ref{Theo:TCrit} that $R \ast G$ can not be endowed with a higher preprojective grading. As seen in \Cref{Ex:Klein-Group-Ungradeable}, the same is true for $G \simeq C_2 \times C_2$. 

Now, suppose that we are in neither of the two cases above. Denote by $\rho = \chi_1 \oplus \chi_2 \oplus \chi_3$ the given embedding of $G$ in $\SL_3(k)$. By \Cref{Lem:Number_Of_Cyclic_Factors} we have that $G$ is cyclic or of the form $C_m \times C_{lm}$. Recall that the assumption of $\rho$ not factoring trough an embedding $\SL_1(k) \times \SL_2(k) \hookrightarrow \SL_3(k)$ means that none of the $\chi_i$ are trivial. Writing this out for the case of $G \simeq C_n$ being cyclic, we find that $G$ has a generator $\frac{1}{n}(e_1, e_2, e_3)$ with all $e_i$ nonzero. As pointed out in \Cref{Rem:SL3-Cyclic-AIR possible}, we can therefore find an AIR-grading on $R \ast G$. If $G$ however is not cyclic, we have by assumption that $|G| > 4$, so we apply \Cref{Lem:Factor_Decomp} to decompose $G \simeq H \times K$ such that $H$ is cylic and none of the $\chi_i \restriction_H$ is trivial. That means we can find an AIR-grading on $R \ast H$. To complete the proof, we apply \Cref{Cor: Lifting strategy} to the AIR-grading on $R \ast H$. The skewed grading of the AIR-grading is therefore a higher preprojective grading of $R \ast G \simeq (R \ast H) \ast K$. 
\end{proof}

\begin{Rem}
One can describe the higher preprojective gradings constructed in the proof of \Cref{Theo:SL3 Classification} explicitly using the McKay quiver as follows. Fix an abelian non-cyclic group $G \leq \SL_3(k)$ for which we know that a higher preprojective grading on $R \ast G$ exists. After decomposing the given embedding into a direct sum of linear characters $\rho = \chi_1 \oplus \chi_2 \oplus \chi_3$ we have that the McKay quiver $Q_\rho$ of $ (G, \rho)$ is the same as the Cayley graph of $(\hat{G}, \{ \chi_1, \chi_2, \chi_3 \})$. Restricting $\rho$ to $H \leq G$ corresponds to passing to the quotient $\hat{G}/H^\bot$. The Cayley graph of $(\hat{G}/H^\bot, \{\chi_1 \cdot H^\bot, \chi_2 \cdot H^\bot, \chi_3 \cdot H^\bot \})$ is obtained by taking the quotient of $Q_\rho$ by the action of $H^\bot$. We therefore have a projection $p \colon Q_\rho \to Q_{\rho \restriction_H}$, and a higher preprojective grading on $Q_\rho$ is the pullback of an AIR-grading on $Q_{\rho \restriction_H}$ along $p$. 
\end{Rem}

\begin{Exp}
    Choose a primitive third root of unity $\zeta$, and consider $g = \frac{1}{3} (2,1,0) $ and $ h= \frac{1}{3}(0,1,2) $. The group $G = \langle g,h \rangle \leq \SL_3(k) $ is abstractly $G \simeq C_3 \times C_3$, so we label its irreducible representations by pairs $\{ (i,j) \mid 0 \leq i,j \leq 2 \}$ so that a pair $(i,j)$ corresponds to the representation $(g^a, h^b) \mapsto \zeta^{ia + jb}$. In particular, the representation given by the embedding decomposes as $\rho = (2,0) \oplus (1,1) \oplus (0,2)$. The McKay quiver is then given by the following quiver, where the repeated vertices and arrows are identified. 
    
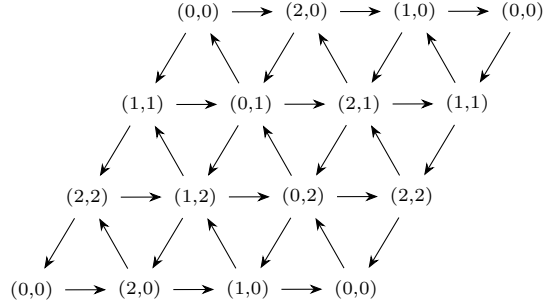
\[ 
    \begin{tikzpicture}[-Stealth]
        \begin{scope}[rotate=-15.4,inner sep=1.5mm]
        \begin{scriptsize}

        \coordinate (O) at (0, 0, 0);
        \coordinate (X) at (1,0,-1);
        \coordinate (Y) at (0, -1, 1);
        \coordinate (Z) at (-1, 1, 0);

        \node(A0) at (O) {(0,0)};
        \node(A1) at (X) {(2,0)};
        \node(A2) at ($2 *(X)$) {(1,0)};
        \node(A3) at ($3 *(X)$) {(0,0)};
        \node(B0) at ($(O) + (Y)$) {(1,1)} ;
        \node(B1) at ($(X) + (Y)$) {(0,1)};
        \node(B2) at ($2 *(X) + (Y)$) {(2,1)};
        \node(B3) at ($3 *(X) + (Y)$) {(1,1)};
        \node(C0) at ($(O) + 2 *(Y)$) {(2,2)} ;
        \node(C1) at ($(X) + 2 *(Y)$) {(1,2)};
        \node(C2) at ($2 *(X) + 2 *(Y)$) {(0,2)};
        \node(C3) at ($3 *(X) + 2 *(Y)$) {(2,2)};
        \node(D0) at ($(O) + 3 *(Y)$) {(0,0)} ;
        \node(D1) at ($(X) + 3 *(Y)$) {(2,0)};
        \node(D2) at ($2 *(X) + 3 *(Y)$) {(1,0)};
        \node(D3) at ($3 *(X) + 3 *(Y)$) {(0,0)};

        \draw(A0) -- (A1);
        \draw(A1) -- (A2);
        \draw(A2) -- (A3);
        \draw(B0) -- (B1);
        \draw(B1) -- (B2);
        \draw(B2) -- (B3);
        \draw(C0) -- (C1);
        \draw(C1) -- (C2);
        \draw(C2) -- (C3);
        \draw(D0) -- (D1);
        \draw(D1) -- (D2);
        \draw(D2) -- (D3);

        \draw(A0) -- (B0);
        \draw(A1) -- (B1);
        \draw(A2) -- (B2);
        \draw(A3) -- (B3);
        \draw(B0) -- (C0);
        \draw(B1) -- (C1);
        \draw(B2) -- (C2);
        \draw(B3) -- (C3);
        \draw(C0) -- (D0);
        \draw(C1) -- (D1);
        \draw(C2) -- (D2);
        \draw(C3) -- (D3);

        \draw(B1) -- (A0);
        \draw(B2) -- (A1);
        \draw(B3) -- (A2);
        \draw(C1) -- (B0);
        \draw(C2) -- (B1);
        \draw(C3) -- (B2);
        \draw(D1) -- (C0);
        \draw(D2) -- (C1);
        \draw(D3) -- (C2);

        \end{scriptsize}
        \end{scope}

    \end{tikzpicture}
    \]
We note that neither of the two obvious cyclic subgroups $\langle g \rangle$ and $\langle h \rangle$ satisfies the conditions to afford an AIR-cut. However, we know from \Cref{Theo:SL3 Classification} that $R \ast G$ must have a higher preprojective cut. Indeed, the subgroup $H$ generated by $(gh)^{-1} = \frac{1}{3} (1,1,1)$ affords an AIR-cut, and has a direct complement $K$ generated by $\frac{1}{3}(1,2,0)$. The relevant orthogonal group is $H^\bot = \langle \frac{1}{3} (1,1,1) \rangle^\bot = \langle (1,2) \rangle$. We take the quotient of the McKay quiver with respect to the action of this group and obtain the quiver we already saw in \Cref{Ex: C3-AIR}.
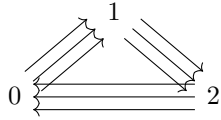
\[\begin{tikzcd}
	& 1 \\
	0 && 2
	\arrow[shift left=2, from=2-1, to=1-2]
	\arrow[shift right=2, from=2-1, to=1-2]
	\arrow[from=2-1, to=1-2]
	\arrow[shift left=2, from=1-2, to=2-3]
	\arrow[shift right=2, from=1-2, to=2-3]
	\arrow[from=1-2, to=2-3]
	\arrow[shift left=2, from=2-3, to=2-1]
	\arrow[shift right=2, from=2-3, to=2-1]
	\arrow[from=2-3, to=2-1]
\end{tikzcd}\]
We choose the same grading as in \Cref{Ex: C3-AIR}, i.e. we place the arrows $2 \to 0$ in degree $1$. The orbits of the $H^\bot$-action on the vertices of the McKay quiver are given by 
\begin{align*}
    (0,0) \cdot H^\bot &= \{ (0,0), (1,2), (2,1) \}, \\
    (1,1) \cdot H^\bot &= \{ (1,1), (2, 0), (0,2) \}, \\
    (2,2) \cdot H^\bot &= \{ (2,2), (0,1), (1,0) \} , 
\end{align*}
so we see that placing the bold arrows in degree $1$ defines a higher preprojective grading on $R \ast G$. 
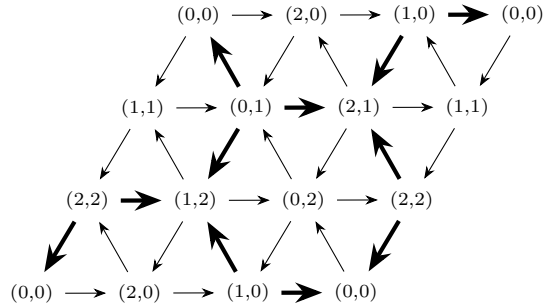
\[ 
    \begin{tikzpicture}[-Stealth]
        \begin{scope}[rotate=-15.4,inner sep=1.5mm]
        \begin{scriptsize}

        \coordinate (O) at (0, 0, 0);
        \coordinate (X) at (1,0,-1);
        \coordinate (Y) at (0, -1, 1);
        \coordinate (Z) at (-1, 1, 0);

        \node(A0) at (O) {(0,0)};
        \node(A1) at (X) {(2,0)};
        \node(A2) at ($2 *(X)$) {(1,0)};
        \node(A3) at ($3 *(X)$) {(0,0)};
        \node(B0) at ($(O) + (Y)$) {(1,1)} ;
        \node(B1) at ($(X) + (Y)$) {(0,1)};
        \node(B2) at ($2 *(X) + (Y)$) {(2,1)};
        \node(B3) at ($3 *(X) + (Y)$) {(1,1)};
        \node(C0) at ($(O) + 2 *(Y)$) {(2,2)} ;
        \node(C1) at ($(X) + 2 *(Y)$) {(1,2)};
        \node(C2) at ($2 *(X) + 2 *(Y)$) {(0,2)};
        \node(C3) at ($3 *(X) + 2 *(Y)$) {(2,2)};
        \node(D0) at ($(O) + 3 *(Y)$) {(0,0)} ;
        \node(D1) at ($(X) + 3 *(Y)$) {(2,0)};
        \node(D2) at ($2 *(X) + 3 *(Y)$) {(1,0)};
        \node(D3) at ($3 *(X) + 3 *(Y)$) {(0,0)};

        \draw(A0) -- (A1);
        \draw(A1) -- (A2);
        \draw[ultra thick](A2) -- (A3);
        \draw(B0) -- (B1);
        \draw[ultra thick](B1) -- (B2);
        \draw(B2) -- (B3);
        \draw[ultra thick](C0) -- (C1);
        \draw(C1) -- (C2);
        \draw(C2) -- (C3);
        \draw(D0) -- (D1);
        \draw(D1) -- (D2);
        \draw[ultra thick](D2) -- (D3);

        \draw(A0) -- (B0);
        \draw(A1) -- (B1);
        \draw[ultra thick](A2) -- (B2);
        \draw(A3) -- (B3);
        \draw(B0) -- (C0);
        \draw[ultra thick](B1) -- (C1);
        \draw(B2) -- (C2);
        \draw(B3) -- (C3);
        \draw[ultra thick](C0) -- (D0);
        \draw(C1) -- (D1);
        \draw(C2) -- (D2);
        \draw[ultra thick](C3) -- (D3);

        \draw[ultra thick](B1) -- (A0);
        \draw(B2) -- (A1);
        \draw(B3) -- (A2);
        \draw(C1) -- (B0);
        \draw(C2) -- (B1);
        \draw[ultra thick](C3) -- (B2);
        \draw(D1) -- (C0);
        \draw[ultra thick](D2) -- (C1);
        \draw(D3) -- (C2);

        \end{scriptsize}
        \end{scope}

    \end{tikzpicture}
    \]
    Removing the arrows in degree $1$ and rearranging the quiver shows that the quiver of the $2$-representation infinite algebra $\Lambda$ we found is $2$-levelled and given by the following. 
    
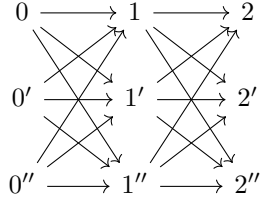
\[\begin{tikzcd}
	0 & 1 & 2 \\
	{0'} & {1'} & {2'} & {} \\
	{0''} & {1''} & {2''}
	\arrow[from=1-1, to=1-2]
	\arrow[from=1-2, to=1-3]
	\arrow[from=1-1, to=2-2]
	\arrow[from=1-1, to=3-2]
	\arrow[from=2-1, to=1-2]
	\arrow[from=2-1, to=2-2]
	\arrow[from=2-1, to=3-2]
	\arrow[from=3-1, to=1-2]
	\arrow[from=3-1, to=2-2]
	\arrow[from=3-1, to=3-2]
	\arrow[from=2-2, to=2-3]
	\arrow[from=1-2, to=2-3]
	\arrow[from=1-2, to=3-3]
	\arrow[from=2-2, to=1-3]
	\arrow[from=2-2, to=3-3]
	\arrow[from=3-2, to=1-3]
	\arrow[from=3-2, to=2-3]
	\arrow[from=3-2, to=3-3]
\end{tikzcd}\]
Let us also note the following. The cyclic group $C_3$ acts on the quiver of $\Lambda$ by permutating the vertices and arrows in each level. One can compute the skew-group algebra $\Lambda \ast C_3$ with respect to this action and see that this is no longer basic, but Morita equivalent to the Beilinson-algebra which is the degree $0$-part of the AIR-cut we saw earlier in this example. Indeed, this is to be expected: We may write $R \ast G = (R \ast H) \ast K$, where $(R \ast H)$ admits an AIR-cut. Skewing by the dual $\hat{K}$ of $K$ in its corresponding dual representation ``unskews'' the action, see \cite[Corollary 5.2]{ReitenRiedtmann}, and we are left with $((R \ast H) \ast K) \ast \hat{K} \simeq (R \ast H) \otimes (\Bbbk K \otimes \Bbbk \hat{K}) \simeq_M (R \ast H)$. All the involved actions and equivalences can be chosen to respect the higher preprojective gradings. 
\end{Exp}

\begin{Rem}\label{Rem: General HPG strategy}
The strategy from \Cref{Cor: Lifting strategy} can be applied to construct $d$-preprojective algebras of type $\Tilde{A}$ for $d \geq 3$. The main obstacle in generalising our classification seems to be that it is unknown for arbitrary cyclic groups when their corresponding skew-group algebras can be endowed with a higher preprojective grading. 
\end{Rem}

\begin{Exp}
    Consider the group $G = \langle \frac{1}{11}(1,1,5,7,8) \rangle \leq \SL_5(k)$ acting on $R = \Bbbk[X_1, \ldots, X_5]$. A direct computer search shows that this is the smallest cyclic group of prime order which does not satisfy the conditions from \Cref{Theo:AIR cuts} and neither those from \Cref{Theo:TCrit}, so we can not directly determine whether $R \ast G$ affords a higher preprojective cut. We label the vertices of the McKay quiver by $\{0, 1, \ldots, 10 \}$ and label the arrows $i \xrightarrow[]{j} i+e_j $ with their type $j$, where $\frac{1}{11}(e_1, e_2, e_3, e_4, e_5) = \frac{1}{11}(1,1,5,7,8)$. If a higher preprojective cut exists, then some arrow of type $1$ needs to be in preprojective degree $1$, since otherwise the infinite-dimensional subalgebra $\Bbbk[X_1 \otimes 1]$ would remain in degree $0$. Without loss of generality, we place the arrow $0 \xrightarrow[]{1} 1$ in degree $1$. Then we consider two $5$-cycles which are in the support of the superpotential for $R \ast G$, given by 
    \begin{align*}
        &0 \xrightarrow[]{1} 1 \xrightarrow[]{2} 2 \xrightarrow[]{3} 7 \xrightarrow[]{4} 3 \xrightarrow[]{5} 0,\\
        &0 \xrightarrow[]{1} 1 \xrightarrow[]{4} 8 \xrightarrow[]{3} 2 \xrightarrow[]{5} 10 \xrightarrow[]{2} 0. 
    \end{align*}
    Since $0 \xrightarrow[]{1} 1$ has degree $1$, we see that $2 \xrightarrow[]{3} 7$ and $8 \xrightarrow[]{3} 2$ must have degree $0$. However, there is a $3$-cycle $c \colon 8 \xrightarrow[]{3} 2 \xrightarrow[]{3} 7 \xrightarrow[]{1} 8$, and from the shape of the superpotential it is clear that this cycle is not nilpotent. Therefore, one of its arrows needs to have preprojective degree $1$, since otherwise the infinite-dimensional algebra $\Bbbk[c]$ would be in degree $0$. Hence, from the assumption that $0 \xrightarrow[]{1} 1$ has degree $1$, we conclude that the same is true for $7 \xrightarrow[]{1} 8$. Iterating the argument we find that all arrows of type $1$ need to be of degree $1$, so we arrive at the conclusion that the grading places $X_1 \otimes 1$ in degree $1$. With this grading, $R \ast G$ already has Gorenstein parameter $1$, but its degree $0$ part is still infinite-dimensional. Hence $R \ast G$ can not be endowed with a higher preprojective cut. 
\end{Exp}

\section{Invariants and mutations}\label{Sec: Tilings}
Given the classification result from \Cref{Theo:SL3 Classification}, it is natural to ask how many different higher preprojective cuts a given skew-group algebra can be endowed with. Note that already in the case of AIR-cuts there is a choice of labelling involved, so different labellings give different cuts. However, the AIR-cuts for different labellings are all equivalent under \emph{cut mutation}. We will recall the mutation procedure and introduce certain mutation invariants. We will show that these invariants are sharp, and classify which values the invariants can take, thus classifying cuts up to mutation. By traversing the mutation class, one can compute all possible higher preprojective cuts, though the number of cuts grows rapidly with the size of $G$. With the necessary theory established, we also give an alternative proof of \Cref{Theo:SL3 Classification} which is independent of \cite{AIR} and \cite{LeMeur}. 

For this section, we consider the case of an abelian group $G \leq \SL_3(k)$ acting on $R = \Bbbk[X_1, X_2, X_3]$ such that $R \ast G$ has a higher preprojective cut. As pointed out in \Cref{SSec:Cayley graphs}, we can decompose the natural representation $\rho = \rho_1 \oplus \rho_2 \oplus \rho_3$ such that $R \ast G \simeq \Bbbk Q /I$ where the quiver $Q$ is the directed Cayley graph with multiplicities for the generating set $\{ \rho_1, \rho_2 , \rho_3 \}$ of $\hat{G}$. To simplify notation, we assume that $|G|=n$, and write $Q_0 = \{ 1, \ldots , n\}$. Further, every arrow $\alpha \colon i \rightarrow j$ in $Q$ is given by an isomorphism $\chi_i \otimes \rho_s \to \chi_j$ for a unique $s \in \{1,2,3 \}$. We will call $\theta(\alpha) = s$ the type of the arrow. Let us now characterise a higher preprojective cut on $R \ast G$ in terms of the corresponding grading on $kQ/I$. Recall that by definition every arrow in Q is homogeneous of degree 0 or 1 for any higher preprojective cut on $kQ/I$, so we can view $Q$ as graded. Furthermore, recall from \Cref{SSec:Cayley graphs} that the superpotential $\omega$ on $R \ast G \simeq kQ/I$ corresponds to a linear combination $\omega_Q$ of all $3$-cycles in $Q$ consisting of arrows of three distinct types. The ideal $I$ generated by the derivatives of $\omega_Q$ is then generated by commutativity relations between arrows of different types, i.e. any path $x_1 \xrightarrow{\alpha_1} x_2 \xrightarrow{\alpha_2} x_3$ consisting of arrows of distinct types $\theta(\alpha_1) \neq \theta(\alpha_2)$ is equal to the path $x_1 \xrightarrow{\beta_2} x_2' \xrightarrow{\beta_1} x_3$ where $\theta(\beta_i) = \theta(\alpha_i)$ for $i \in \{1,2\}$.  
  
\begin{Rem}
    With the above setup, we can identify a higher preprojective cut on $R \ast G \simeq \Bbbk Q/I$ with the set $C = \{ \alpha \in Q_1 \mid \deg(\alpha) = 1 \}  \subseteq Q_1 $ of arrows of higher preprojective degree $1$, which we subsequently also call a cut. 
\end{Rem}

\begin{Def}
    Fix a subset $C \subseteq Q_1$. We call the quiver $Q_C = (Q_0, Q_1 - C)$ with arrows from $C$ removed the \emph{cut  quiver} with respect to $C$. The cut algebra is then $(kQ/I)_C = kQ_C/(I \cap kQ_C)$.
\end{Def}

\begin{Def}\label{Def: Type of a Cut}
    With the above setup, let $C \subseteq Q_1$ be a higher preprojective cut. The triple 
    \[ \theta(C) = (\theta_1(C), \theta_2(C), \theta_3(C) ) = ( \# \{ \alpha \in C \mid \theta(\alpha) = s  \} )_{s \in \{ 1,2,3\} } \]
    is called the \emph{type} of $C$. 
\end{Def}

With this terminology, we now prove that a set $C \subseteq Q_1$ is a higher preprojective cut if and only if all $3$-cycles have exactly one arrow in $C$, and $C$ contains arrows of all types. See also \Cref{Rem: Tension between dim and GP} for how the two conditions interact. Note that we give a combinatorial proof of one of the implications in \Cref{Lem: Positive type is acyclic}. 

\begin{Pro} \label{Pro: HPPCut iff positive and homogeneous}
	A subset $C \subseteq Q_1$ is a higher preprojective cut if and only if the following two conditions are satisfied: 
\begin{enumerate}
	\item For every $3$-cycle $\alpha_1 \alpha_2 \alpha_3$ in $Q$ such that $\{ \theta(\alpha_i) \mid 1 \leq i \leq 3 \} = \{1,2,3\}$ there exists exactly one $i \in \{1,2,3\}$ such that $\alpha_i \in C$. 
    \item For every type $s \in \{1,2,3 \}$, there exists an arrow $\alpha \in C$ such that $\theta(\alpha) = s$. 
\end{enumerate}
\end{Pro}

\begin{proof}
Recall that the superpotential $\omega_Q$ for $kQ/I$ is a linear combination of all $3$-cycles in $Q$ consisting of arrows of three distinct types. Suppose $C$ is a higher preprojective cut. Then by \Cref{Pro: Potential is homogeneous}, the superpotential $\omega_Q$ needs to be homogeneous of degree 1. Hence, each $3$-cycle $\alpha_1 \alpha_2 \alpha_3$ in the support of $\omega_Q$ is homogeneous of degree $1$, which means each such $3$-cycle has exactly one arrow in $C$, so condition (1) is satisfied. Furthermore, if $C$ is higher preprojective, then $(R \ast G)_0 \simeq (kQ/I)_0 = (kQ)_C$ is finite-dimensional. Denote by $x_i$ the sum of all arrows of type $i$. Then $k[x_i]$ is a polynomial ring, hence infinite-dimensional, and therefore can not be contained in $(kQ/I)_0 = (kQ)_C$, which means that $x_i$ can not be homogeneous of degree $0$. Thus, some of the arrows of type $i$ need to be of degree $1$, which means they are in $C$. Since $i$ was arbitrary, $C$ contains arrows of all types, so condition (2) is satisfied.

Conversely, suppose that $C$ satisfies conditions (1) and (2). It is easy to see that placing the arrows in $C$ in degree $1$ defines a grading on $kQ/I$. From condition (1) it follows immediately that $\omega_Q$ becomes homogeneous of degree $1$ for the induced grading on $kQ$, and then it follows from \cite[Corollary 4.7]{Giovannini} that the grading on $kQ/I$ has Gorenstein parameter $1$. It remains to show that $(kQ/I)_0  = (kQ)_C$ is finite-dimensional. For a contradiction, suppose that $(kQ)_C$ is infinite-dimensional. Since there are only finitely many non-cyclic paths in $Q$, it follows that there must be a cycle $c$ in $(kQ)_C$, so all of the arrows in $c$ have degree $0$. Recall that the ideal $I$ of relations for $kQ/I$ is generated by commutativity relations. Using these relations, if $c$ contains arrows of all three types, we may use the relations to rewrite $c$ in $kQ/I$ as a cycle containing a $3$-cycle consisting of three arrows of three distinct types, so by condition (1) we see that $c$ must have at least degree $1$, a contradiction to the assumption that $c$ has degree $0$. If $c$ is of degree $0$ and contains only arrows of two types, we assume without loss of generality that the types are $1$ and $2$. Then we use the relations to rewrite $c$ as a cycle $\alpha_1 \cdots \alpha_{i} \alpha_{i+1} \cdots \alpha_l$ so that the first $i$ arrows all have type $1$, and the last $l-i$ all have type $2$. It then follows from condition (1) that the arrow of type $3$ completing $\alpha_i \alpha_{i+1}$ into a $3$-cycle must have degree $1$. This then implies that the two arrows $\beta_{i+1}$, $\beta_{i}$ of types $2$ and $1$ respectively with $\alpha_{i} \alpha_{i+1} = \beta_{i+1} \beta_{i} $ both have degree $0$, meaning they lie in $(kQ)_C$, so we write $c =  \alpha_1 \cdots \alpha_{i-1} \beta_{i+1} \beta_{i}  \alpha_{i+2} \cdots \alpha_l$ in $(kQ)_C$, and using the same argument we now see that the arrows of type $3$ completing $\alpha_{i-1} \beta_{i+1}$ and $\beta_{i} \alpha_{i+2}$ into $3$-cycles must have degree $1$. It follows inductively that indeed all arrows of type $3$ have degree $1$, and then condition (1) shows that no arrows of type $1$ or $2$ are in $C$, which contradicts condition (2). Lastly, if $c = \alpha_1 \cdots \alpha_l$ contains only arrows of one type, we assume without loss of generality that the type is $1$. Consider the arrows $\beta_2$, $\beta_3$ of types $2$ and $3$ forming a $3$-cycle $\alpha_1 \beta_2 \beta_3$. Then by condition (1), one of the arrows has degree $1$, and we assume without loss of generality that $\beta_2$ has degree $1$. Then the arrows $\alpha_1'$ and $\beta_3'$ of types $1$ and $3$ forming a $3$-cycle $\alpha_1' \beta_3' \beta_2$ must also be in degree $0$. But now we see that the arrow $\beta_2''$ of type $2$ completing $\beta_3' \alpha_2$ into a $3$-cycle must also be of degree $1$. We denote by $c' = \alpha_1' \cdots \alpha_l'$ the cycle consisting of arrows of type $1$ and containing $\alpha_1'$, and note that none of the $\alpha_i'$ are in $c$. It then follows from the degree of $\beta_2''$ that $\alpha_2'$ is of degree $0$, and inductively $c'$ is homogeneous of degree $0$. Repeating the argument for $c'$, we see inductively that every arrow of type $1$ has degree $0$, so it follows from condition (1) that no arrow of type $2$ or $3$ has degree $1$, contradicting condition (2). 
\end{proof}

Most of our results concerning the classification, construction and mutation of cuts are based on periodic lozenge tilings of the plane. We state the results here and postpone the proofs using tilings explicitly to \Cref{SSec: Class of types} and \Cref{SSec: Mutation and flips}. We will need the following data throughout this section, coming from the structure theorem of finitely generated abelian groups.

\begin{Pro} \label{Pro: Periodicity matrix}
    There exists an integer matrix $B = (\begin{smallmatrix}    a_1 & b_1 \\ a_2 & b_2 \end{smallmatrix} ) \in \mathbb{Z}^{2 \times 2}$ such that $\rho_1^{a_1} \rho_2^{a_2} = \bbone =\rho_1^{b_1} \rho_2^{b_2}$ and $\det(B) = n$. The matrix is unique up to right-multiplication by $\SL_2(\mathbb{Z})$. 
\end{Pro}

\begin{proof}
    Recall from \Cref{Lem:Faithful_Iff_Generating} that $\hat{G}$ is generated by $\{ \rho_1, \rho_2\}$. Consider the associated presentation 
    \[ 0 \to \mathbb{Z}^2 \xrightarrow[]{\iota} \mathbb{Z}^2 \xrightarrow[]{\varphi} \hat{G} \to 0. \]
    Choosing appropriate bases for both terms $\mathbb{Z}^2$, we obtain a matrix $B$ for $\iota$ which satisfies $\rho_1^{a_1} \rho_2^{a_2} = \bbone =\rho_1^{b_1} \rho_2^{b_2}$ and $\det(B) = n$. Clearly, $B$ is unique up to ordered base change of $\Ker(\varphi)$, i.e. up to right-multiplication by $\SL_2(\mathbb{Z})$. 
\end{proof}

\begin{Rem}\label{Rem: Positive matrix}
   The matrix $B$ in \Cref{Pro: Periodicity matrix} can be chosen so that $a_1,b_2>0$, $a_2=0$, $b_1\ge 0$. More precisely, we can choose a basis $e_1, e_2 \in \mathbb{Z}^2$ such that $\varphi(e_i) = \rho_i$, and analyzing the relations shows that there exists a basis $f_1, f_2 \in \mathbb{Z}^2$ such that $\iota(f_1) = |\rho_1| e_1 $ and $\iota(f_2) =  b_1e_1 + \frac{|G|}{|\rho_1|} e_2  $, where $b_1 = \min\{ m \geq 0 \mid \varphi(me_1 + \frac{|G|}{|\rho_1|} e_2 ) = 0 \} $. 
\end{Rem}

Fix a matrix $B$ as in \Cref{Pro: Periodicity matrix}. We will prove the following theorem, see \Cref{Theo: Divisibility conditions for tilings}.  

\begin{Theo}\label{Theo: Divisibility conditions}
    The type of any higher preprojective cut $C$ satisfies $\theta_1(C) + \theta_2(C) + \theta_3(C) = n $ and $\theta_i(C) > 0$ for $1 \leq i \leq 3$. Furthermore, a triple $(\gamma_1, \gamma_2, \gamma_3) \in \mathbb{N}_{> 0}^{1 \times 3}$ with $\gamma_1 + \gamma_2 + \gamma_3 = n$ is the type of a higher preprojective cut if and only if we have 
    \[ (\gamma_1, \gamma_2) B \in n \mathbb{Z}^{1 \times 2}. \]
\end{Theo}

Let us illustrate this theorem with an example, showing that this allows for easy computation of all types. 
\begin{Exp}\label{Exp: size 12 example with simplex}
    Consider the group $G = \langle \frac{1}{2}(1,1,0), \frac{1}{6}(1,4,1) \rangle$, which is noncyclic of order $n = 12$. Furthermore, it is easy to check that a matrix $B$ as in \Cref{Pro: Periodicity matrix} is given by $(\begin{smallmatrix}
        6 & 4 \\ 0 & 2 
    \end{smallmatrix}).$ We can then give all possible types of higher preprojective cuts as follows. Denote the type by $(\gamma_1, \gamma_2, \gamma_3)$, where for each $i$ we have $0 < \gamma_i < 12$, and compute 
    \[ (\gamma_1, \gamma_2) B = (6 \gamma_1, 4 \gamma_1 + 2 \gamma_2).  \]
    To satisfy the condition from \Cref{Theo: Divisibility conditions}, we need $12 \mid 6\gamma_1$ and $12 \mid (4 \gamma_1 + 2 \gamma_2)$ which simplifies to $6 \mid (2\gamma_1+\gamma_2)$. The first divisibility condition gives $\gamma_1 \in \{ 2,4,6,8,10 \} $. Taking each possibility for $\gamma_1$, we substitute $\gamma_1$ in the second condition. Keeping in mind that $\gamma_1 + \gamma_2 + \gamma_3 = 12$, we see immediately that $\gamma_1 \in \{ 6, 10 \}$ can not occur for the type of a higher preprojective cut. In the other cases, all positive solutions to the second condition are given by 
    \begin{align*}
        (4,4,4), (8, 2, 2), (2,8,2), (2,2,8). 
    \end{align*}
    It is convenient to visualise these points in the simplex $\{ (\gamma_1, \gamma_2, \gamma_3) \in \mathbb{N}^{1 \times 3} \mid \gamma_1 + \gamma_2 + \gamma_3 = 12  \}$, and mark the types of higher preprojective cuts in red, as well as those points in the simplex which have a zero coordinate but still satisfy the condition from \Cref{Theo: Divisibility conditions}. These extra points are marked in black and have coordinates $(12, 0, 0), (0, 12, 0), (0,0,12), (6,6,0), (6,0,6), (0, 6,6).$
    
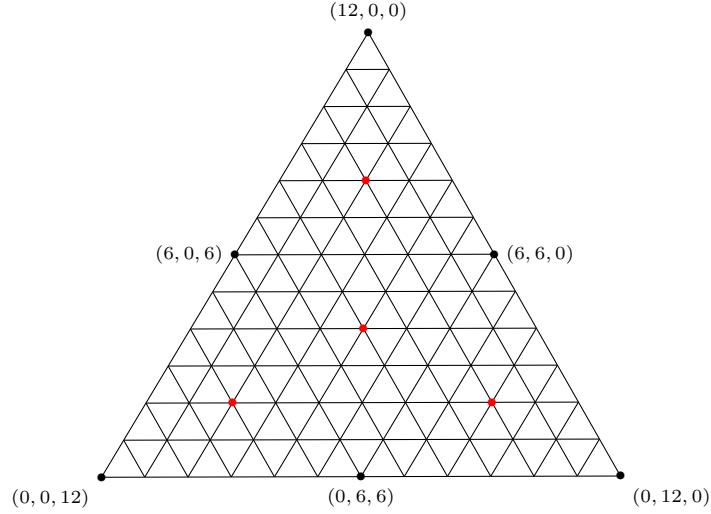
\[
\begin{tikzpicture}[-,scale=0.4]

\begin{scope}[rotate=-15.4,inner sep=1.5mm]

\begin{scriptsize}

\coordinate (O) at (0, 0, 0);
\coordinate (X) at (1,0,-1);
\coordinate (Y) at (0, -1, 1);
\coordinate (Z) at (-1, 1, 0); 

\foreach \y in {0, ..., 11}
{  
    \pgfmathparse{int(11 - \y)}
    \foreach \x in {0, ..., \pgfmathresult} 
    {   
        \draw ($\x *(X) - \y *(Y) $) -> ($ \x *(X) + (X) - \y *(Y) $);
    }
}

\foreach \x in {0, ..., 11}
{  
    \pgfmathparse{int(11 - \x)}
    \foreach \y in {0, ..., \pgfmathresult} 
    {   
        \draw ($\x *(X) - \y *(Y) $) -> ($ \x *(X)  - \y *(Y) - (Y) $);
    }
}

\foreach \x in {1, ..., 12}
{  
    \foreach \z in {1, ..., \x} 
    {   
        \draw ($\x *(X) + \z *(Z) - (Z)$) -> ($ \x *(X)  + \z *(Z) $);
    }
}

\node[anchor= north east] at (O) {$(0, 0, 12)$};
\node[anchor= north west] at ($ 12 *(X)$) {$(0, 12, 0)$};
\node[anchor= south] at ($-12 *(Y)$) {$(12,0,0)$};

\node[anchor= north] at ($6 *(X)$) {$(0, 6, 6)$};
\node[anchor= east] at ($-6 *(Y)$) {$(6, 0, 6)$};
\node[anchor= west] at ($6 *(X) - 6 *(Y)$) {$(6, 6, 0)$};

\node at ($2 *(X) - 8 *(Y)$) {$\color{red} \bullet$};
\node at ($4 *(X) - 4 *(Y)$) {$\color{red} \bullet$};
\node at ($2 *(X) - 2 *(Y)$) {$\color{red} \bullet$};
\node at ($8 *(X) - 2 *(Y)$) {$\color{red} \bullet$};

\node at ($6 *(X)$) {$\bullet$};
\node at ($-6 *(Y)$) {$\bullet$};
\node at ($6 *(X) - 6 *(Y)$) {$\bullet$};

\node at (O) {$\bullet$};
\node at ($-12 *(Y)$) {$\bullet$};
\node at ($12 *(X)$) {$\bullet$};

\end{scriptsize}
\end{scope}
\end{tikzpicture}
\]
It is not difficult to see that all the marked points (both red and black) lie on an affine lattice. More precisely, after choosing e.g. $(0,0,12)$ as a starting point, all other marked points can be reached by integer combinations of $6u$ and $4u+2v$, where $u^\top = (1,0)$ and $v^\top = -(\frac{1}{2}, \frac{\sqrt{3}}{2}) $ are two vectors defining the sides of the simplex as it is projected onto the page. Note that the appearing coefficients $(6,0)$ and $(4, 2)$ are precisely the columns of the matrix $B$ we started with.
\end{Exp}

The previous example suggests an alternative method of computing all types of higher preprojective cuts. Fix a matrix $B = ( \begin{smallmatrix} a_1 & b_1 \\ a_2 & b_2 \end{smallmatrix})$ as above, as well as the vectors $u^\top = (1,0)$ and $v^\top = -(\frac{1}{2}, \frac{\sqrt{3}}{2}) $, and the lattices $L_0 = \langle u,v \rangle$ and $L_1 = \langle a_1u + a_2v, b_1 u + b_2v \rangle$. Then the $L_1$-points in the interior of an appropriately chosen simplex correspond to the types of higher preprojective cuts. This is the content of the following proposition, which is equivalent to \Cref{Theo: Divisibility conditions}.

\begin{Pro}\label{Pro: Simplex points are types}
With the notation from above, consider $p \colon \mathbb{Z}^{1 \times 3} \to L_0, (x,y,z) \mapsto (yu - xv)$. Denote by $\Delta_n =  \{ (\gamma_1, \gamma_2, \gamma_3) \in \mathbb{N}_{> 0}^{1 \times 3} \mid \gamma_1 + \gamma_2 + \gamma_3 = n  \}$ the standard simplex of potential types. Then $(\gamma_1, \gamma_2, \gamma_3) \in \Delta_n$ is the type of a higher preprojective cut if and only if $p((\gamma_1, \gamma_2, \gamma_3))$ is an $L_1$-lattice point in $p(\Delta_n)$, i.e. $p((\gamma_1, \gamma_2, \gamma_3))$ is the type of a higher preprojective cut if and only if 
\[ p((\gamma_1, \gamma_2, \gamma_3)) \in p(\Delta_n)  \cap L_1. \]    
\end{Pro}

\begin{proof}
     If $(\gamma_1,\gamma_2,\gamma_3)$ is a type of a higher preprojective cut, then $p((\gamma_1, \gamma_2, \gamma_3)) \in p(\Delta_n)$. Furthermore, by \Cref{Theo: Divisibility conditions}, we have that $\frac{\gamma_1a_1+\gamma_2a_2}{n},\frac{\gamma_1b_1+\gamma_2b_2}{n}\in \mathbb{Z}$. Then 
     \begin{align*}
    &\frac{\gamma_1b_1+\gamma_2b_2}{n}(a_1u+a_2v)-\frac{\gamma_1a_1+\gamma_2a_2}{n}(b_1u+b_2v)\\
    =&\frac{a_1\gamma_1b_1+a_1\gamma_2b_2-b_1\gamma_1a_1-b_1\gamma_2a_2}{n}u+\frac{a_2\gamma_1b_1+a_2\gamma_2b_2-b_2\gamma_1a_1-b_2\gamma_2a_2}{n}v\\
    =&\frac{a_1\gamma_2b_2-b_1\gamma_2a_2}{n}u+\frac{a_2\gamma_1b_1-b_2\gamma_1a_1}{n}v=\gamma_2u-\gamma_1v.
     \end{align*}
     
Conversely, an arbitrary point in $L_1$ is of the form \[ c_1(a_1u+a_2v)-c_2(b_1u+b_2v)=(c_1a_1-c_2b_1)u-(c_2b_2-c_1a_2)v ,\]
which has been projected from 
\[ (\gamma_1, \gamma_2, \gamma_3) =  (c_2b_2-c_1a_2,c_1a_1-c_2b_1, n-(c_2b_2-c_1a_2+c_1a_1-c_2b_1)).\] Then we have
\begin{align*}
    \gamma_1 a_1 + \gamma_2 a_2 &= a_1(c_2b_2-c_1a_2)+a_2(c_1a_1-c_2b_1)=c_2(a_1b_2-a_2b_1)=c_2n \\
    \gamma_1 b_1 + \gamma_2 b_2 &= b_1(c_2b_2-c_1a_2)+b_2(c_1a_1-c_2b_1)=c_1(a_1b_2-a_2b_1)=c_1n.
\end{align*}  
Hence, by \Cref{Theo: Divisibility conditions}, if this point is in $p(\Delta_n) \cap L_1$, then $(\gamma_1, \gamma_2, \gamma_3)$ is the type of a higher preprojective cut. 
\end{proof}

Next, we consider mutation. 

\begin{Def}\label{Def: Mutation of cuts}
    Given a higher preprojective cut $C \subseteq Q_1$, we call a vertex $v \in Q_0$ \emph{mutable} (in $Q_C$) if it is a source or a sink in $Q_C$. If $v$ is a source, we obtain a new cut 
    \[ \mu_v(C) = \{ \alpha \in C \mid t(\alpha) \neq v \} \cup \{ \alpha \in Q_1 \mid s(\alpha) = v   \},   \]
    called the \emph{mutation of $C$ at $v$}. \newline 
    The mutation at a sink $v$ in $Q_C$ is defined dually as
    \[\mu_v(C) = \{ \alpha \in C \mid s(\alpha) \neq v \} \cup \{ \alpha \in Q_1 \mid t(\alpha) = v   \}. \]
    A \emph{mutation sequence} in $Q$ is a sequence of vertices $v_1, \ldots, v_m \in Q_0$ such that $v_i$ is mutable in $Q_{(\mu_{i-1} \circ \cdots \circ \mu_{1})(C)}$. 
\end{Def}

\begin{Rem}
    A direct consequence of the definition is that a mutable vertex $v$ stays mutable after mutation, and that $\mu_v^2(C) = C$. Furthermore, since every vertex has exactly one incoming and one outgoing arrow of each type, we see that the type of a cut $C$ is invariant under mutation, i.e. for any mutable vertex $v$ in $Q_C$ we have 
    \[ \theta(C) = \theta(\mu_v(C)). \]
\end{Rem}

We will prove that the type is a sharp mutation invariant, see \Cref{Theo: Flips are transitive}. 

\begin{Theo}\label{Theo: Mutation is transitive}
    Two cuts $C_1$, $C_2$ have the same type if and only if they are connected by a mutation sequence, i.e. they fulfill $\theta(C_1) = \theta(C_2)$ if and only if there exists a sequence of vertices $v_1, \ldots, v_m \in Q_0$ such that $(\mu_{v_m} \circ \cdots \circ \mu_{v_1})(C_1)  = C_2$. 
\end{Theo}

\begin{Rem}
    The above defined mutation is of interest to us because it corresponds to $2$-APR tilting of the $2$-representation infinite algebra $\Bbbk Q_C/(I \cap \Bbbk Q_C) $ as defined in \cite[Definition 3.1]{IyamaOppermann}. It is known that this tilting preserves the property of being $2$-representation infinite, which we can see directly from the fact that mutation produces new cuts. In particular, this means that if $\theta(C_1) = \theta(C_2)$, then the corresponding algebras are derived equivalent. 
    
    Furthermore, $2$-APR tilting also preserves the property of being $n$-representation finite, and in \cite{IyamaOppermann} it is shown that $n$-representation finite algebras of type $A$ with the same higher preprojective algebra are connected by a sequence of tilts. This is not true in our case, but the different tilting equivalence classes are parametrised by the types. 
\end{Rem}

\begin{Rem}
    Let us note how the above leads to a classification of all $2$-representation infinite algebras of type $\tilde{A}$. One can first list all possible $3$-preprojective algebras using \Cref{Theo:SL3 Classification}, then for a given higher preprojective algebra all types of the possible cuts using \Cref{Theo: Divisibility conditions} or \Cref{Pro: Simplex points are types}. The final step can be done by choosing a type and starting with the cut constructed in the proof of \Cref{Theo: Divisibility conditions}, see \Cref{tilingconstr} for details, and iterating mutations. \Cref{Theo: Mutation is transitive} guarantees that we find all algebras of the chosen type this way. 
\end{Rem}

We also obtain the following corollary. 

\begin{Cor}\label{Cor: Tilting equiv to skewed AIR}
    Let $\Lambda$ be a $2$-representation infinite algebra of type $\tilde{A}$. Then $\Lambda$ is $2$-APR tilting equivalent to an algebra $\Lambda' \ast C_m$ for some $2$-representation infinite algebra $\Lambda'$ of type $\tilde{A}$ arising from an AIR-cut.  
\end{Cor}

\begin{proof}
    Follows from \Cref{tilingconstr} and \Cref{Theo: Mutation is transitive}. 
\end{proof}

\subsection{Universal cover and lozenges}\label{SS: Universal cover}
To prove the statements of this section, we will use methods from the theory of lozenge tilings, which is dual to the theory of hexagonal dimer models on the torus. 

We begin by noting that the quiver of the algebra $R \ast G \simeq \Bbbk Q /I$ is naturally embedded on a torus. More precisely, fix the vectors $u^\top = (1,0)$ and $v^\top = -(\frac{1}{2}, \frac{\sqrt{3}}{2}) $ and consider the lattice $L_0  = (u|v)\mathbb{Z}^2$ of equilateral triangles spanned by these two vectors. Fix also the vector $w = -(u+v)$.  

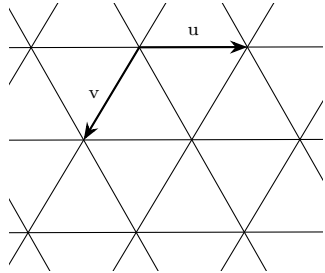
\[
\begin{tikzpicture}[-,scale=1]
\clip (-2.3, 0, -1.5) rectangle ( 5.4,0,7.7);

\begin{scope}[rotate=-15.4,inner sep=1.5mm]

\begin{scriptsize}

\coordinate (O) at (0, 0, 0);
\coordinate (X) at (1,0,-1);
\coordinate (Y) at (0, -1, 1);
\coordinate (Z) at (-1, 1, 0); 

\foreach \x in {-2, ..., 2}
{
    \foreach \y in {-2, ..., 2}
    {
        \draw ($ \x *(X) + \y *(Y) $) -> ($\x *(X) + \y *(Y) + (X)$);
        \draw ($\x *(X) + \y *(Y) + (X)$) -> ($\x *(X) + \y *(Y) + (X) + (Y)$); 
        \draw ($\x *(X) + \y *(Y) + (X) + (Y)$) -> ($ \x *(X) + \y *(Y) $);
        
        \draw ($ -2 *(X) + \y *(Y) $) -> ($ -2 *(X) + \y+1 *(Y) $);
    }
    \draw  ($ \x *(X) + 3 *(Y) $) -> ($\x *(X) + 3 *(Y) + (X)$);
}

\draw ($0.5 *(X)$) node[anchor=south] {u};
\draw ($0.5 *(Y)$) node[anchor=east] {v};
\draw[-Stealth, thick] (0,0,0)   -- (X);
\draw[-Stealth, thick] (0,0,0)   -- (Y);

\node at (-3, 0, 3) {$\cdots$};
\node at ($4*(X)$) {$\cdots$};

\end{scriptsize}
\end{scope}
\end{tikzpicture}
\]

Then we place the irreducible characters of $G$ on the points of $L_0$ so that the actions of $u$ and $v$ correspond to multiplication by $\rho_1$ and $\rho_2$ respectively, i.e. we choose the group morphism $\varphi \colon L_0 \to \operatorname{Irr}(G), u \mapsto \rho_1, v \mapsto \rho_2$.

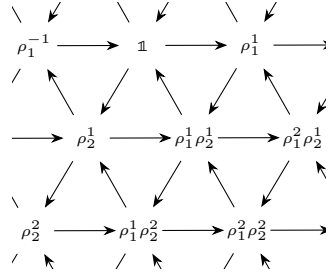
\[
\begin{tikzpicture}[-Stealth,scale=1]
\clip (-2.3, 0, -1.5) rectangle ( 5.4,0,7.7);

\begin{scope}[rotate=-15.4,inner sep=3mm]
\begin{tiny}

\coordinate (O) at (0, 0, 0);
\coordinate (X) at (1,0,-1);
\coordinate (Y) at (0, -1, 1);
\coordinate (Z) at (-1, 1, 0); 

\foreach \x in {-2, ..., 2}
{
    \foreach \y in {-2, ..., 2}
    {
    
        \node (A) at ($ \x *(X) + \y *(Y) $) {};
        \node (B) at ($\x *(X) + \y *(Y) + (X)$) {};
        \node (C) at ($\x *(X) + \y *(Y) + (X) + (Y)$) {};
        \draw (A) -- (B);
        \draw (B) -- (C); 
        \draw (C) -- (A);
        
    }
}

\foreach \x in {-2, ..., 3}
{
    \foreach \y in {-2, ..., 3}
    {
    \ifnum \x=0 
        \ifnum \y=0 \node at (0,0,0) {$\bbone$};
        \else \node at ($\y *(Y)$) {$\rho_2^{\y}$}; 
        \fi
    \else 
        \ifnum \y=0 \node at ($\x *(X)$) {$\rho_1^{\x}$};
        \else
        \node at ($\x *(X) + \y *(Y)$) {$\rho_1^{\x} \rho_2^{\y} $ };
        \fi
    \fi 
    }

}

\end{tiny}
\end{scope}
\end{tikzpicture}
\]

Note that the characters repeat, and we see that the kernel $L_1 = \Ker(\varphi)$ is a cofinite sublattice $L_1  \leq L_0$. We choose some generators for this sublattice, i.e. we write $L_1 = \langle a_1u + a_2v, b_1u + b_2v \rangle $ and record this choice of integer coefficients in the matrix $B = (\begin{smallmatrix} a_1 & b_1 \\ a_2 & b_2 \end{smallmatrix})$. This is, up to permutation of columns, a matrix as in in \Cref{Pro: Periodicity matrix}.

\begin{Def}
    We denote by $\hat{Q}$ the quiver with vertices $\hat{Q}_0 = L_0$ and arrows $\hat{Q}_1 = \{ p \to (p + x) \mid p \in L_0, x \in \{u,v,w\}  \} $. 
\end{Def}

Note that we obtained a covering $q \colon \hat{Q} \to Q$, induced by $L_0 \to L_0/L_1$. We transfer our terminology along this covering. For the sake of visualisation we will refer to arrows of type $u$, $v$ and $w$ in $\hat{Q}$ corresponding to the types $1$, $2$ and $3$ in $Q$ under the quotient map $q$. See also \cite[Construction 5.2]{HIO} for this construction in more generality. 

Given a higher preprojective cut $C$ on $Q$, we consider the set $\hat{C} = q^{-1}(C) = \{ \alpha \in \hat{Q}_1 \mid q(\alpha) \in C  \}$ and note that every triangle in $\hat{Q}$ must have exactly one arrow in $\hat{C}$. Hence, when removing $\hat{C}$ from $\hat{Q}$, we merge pairs of adjacent triangles into \emph{lozenges}. 

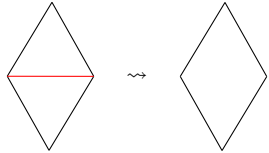
\[
\begin{tikzpicture}[-,scale=0.8]
\begin{scope}[rotate=-15.6,inner sep=1.5mm]
\begin{scriptsize}

\coordinate (O) at (0, 0, 0);
\coordinate (X) at (1,0,-1);
\coordinate (Y) at (0, -1, 1);
\coordinate (Z) at (-1, 1, 0); 

\draw (O) edge (Y);
\draw (Y) edge ($(Y) - (Z)$);
\draw (O) edge ($-1 *(Z)$);
\draw ($-1 *(Z)$) edge ($(Y) - (Z)$);
\draw (Y) edge[color = red] ($-1 *(Z)$);

\draw ($2 *(X)$) -- ($2 *(X) + (Y)$);
\draw ($2 *(X) + (Y)$) edge ($2 *(X) + (Y) - (Z)$);
\draw ($2 *(X)$) edge ($2 *(X) - (Z)$);
\draw ($2 *(X) - (Z)$) edge ($2 *(X) - (Z) + (Y)$);

\node at ($(X) - 0.5 *(Z) + 0.5 *(Y)$) {$\rightsquigarrow$};

\end{scriptsize}
\end{scope}
\end{tikzpicture}
\]
In fact, we obtain a lozenge tiling of the plane, which is periodic with respect to the $L_1$-action. The lozenges come in three different types, corresponding to the type of the removed arrow, so we label them by $U$, $V$ and $W$ accordingly. 
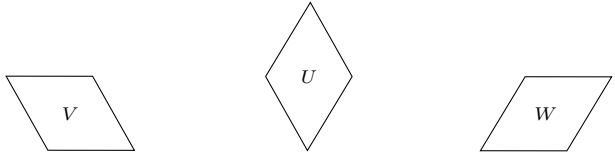
\[
\begin{tikzpicture}[-,scale=0.8]
\begin{scope}[rotate=-15.6,inner sep=1.5mm]
\begin{scriptsize}


\coordinate (O) at (0, 0, 0);
\coordinate (X) at (1,0,-1);
\coordinate (Y) at (0, -1, 1);
\coordinate (Z) at (-1, 1, 0); 

\VLozenge{($-3 *(X)$)}
\ULozenge{(Z)}
\WLozenge{($2 *(X)$)}


\node at ($-2.5 *(X) + 0.5 *(Z) $) {$V$};
\node at ($0.5 *(Z) - 0.5 *(Y)$) {$U$};
\node at ($2.5 *(X) - 0.5 *(Y)$) {$W$};

\end{scriptsize}
\end{scope}
\end{tikzpicture}
\]

It will be useful to broaden our definition of a cut slightly to include those sets of arrows which correspond to periodic lozenge tilings but do not necessarily give rise to higher preprojective cuts. 

\begin{Def}
    An \emph{$L_1$-periodic cut} of $\hat{Q}$ is a set of arrows $\hat{C} \subseteq \hat{Q}_1$ such that removing the arrows in $\hat{C}$ gives an $L_1$-periodic lozenge tiling of the plane. We drop the dependence on $L_1$ when no confusion is possible. 
\end{Def}

\begin{Rem}
    Every higher preprojective cut $C$ on $Q$ gives rise to a periodic cut $q^{-1}(C)$ on $\hat{Q}$. However, there are periodic cuts on $\hat{Q}$ whose associated tilings do not contain lozenges of all three types. Such a periodic cut will induce a grading on $\Bbbk Q/I \simeq R \ast G$ which has Gorenstein parameter $1$, but the degree $0$ piece will be infinite-dimensional, hence it will not be a higher preprojective cut. 
\end{Rem}

It is important to note the following. 

\begin{Pro}\label{Pro: Periodic cuts and HPPC}
    There is a bijection between higher preprojective cuts $C$ on $Q$ and $L_1$-periodic cuts $\hat{C}$ on $\hat{Q}$ in which every type of arrow occurs, given by $C \mapsto \hat{C}$ and $\hat{C} \mapsto q(\hat{C})$.
\end{Pro}

\begin{proof}
    Clearly, every higher preprojective cut $C$ on $Q$ gives rise to an $L_1$-periodic of $\hat{Q}$. Recall from \Cref{Pro: HPPCut iff positive and homogeneous} that for a higher preprojective cut $C$ every type of arrow must occur in $C$, hence every type of arrow appears in $\hat{C}$. Conversely, every $L_1$-periodic cut $\hat{C}$ gives rise to a cut $C = q(\hat{C})$ on $Q$. Interpreting $\hat{C}$ as a lozenge tiling and using again \Cref{Pro: HPPCut iff positive and homogeneous}, we see that $(kQ)_C$ is finite-dimensional, hence $C$ is higher preprojective. 
\end{proof}

\subsection{Classification of types}\label{SSec: Class of types}
The results of this subsection were motivated by the unpublished work \cite{Martin} where classification of types has been done in terms of fundamental triangles. Here we will state a result that is equivalent to \cite[Theorem 6]{Martin}, and prove it using simpler methods and without introducing the fundamental triangle. 

For the rest of this section, let $B = (\begin{smallmatrix} a_1 & b_1 \\ a_2 & b_2 \end{smallmatrix})$ be a $2$-rank integer matrix as in \Cref{Pro: Periodicity matrix}, and let $L$ be the lattice defined by $B$. Let $n=\det B=|L_0/L_1|$. 

A lozenge tiling of $L_0$ can be seen as a subset of its arrows satisfying a certain property. We can also view it as a perfect matching between upwards and downwards pointing triangles ($\triangle$ and \rotatebox[origin=c]{180}{$\triangle$}, respectively). Alternatively, one can view a lozenge tiling as a map $T \colon L_0 \to \{U, V, W\}$ which places a lozenge $T(x)$ at point $x$ as shown below. The point $x$ is the lower left corner of an upwards pointing triangle, and $T(x)$ specifies into which type of lozenge the triangle is merged. 

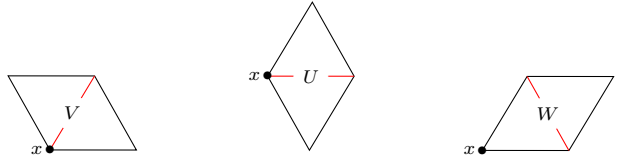
\[
\begin{tikzpicture}[-,scale=0.8]
\begin{scope}[rotate=-15.6,inner sep=1.5mm]
\begin{scriptsize}


\coordinate (O) at (0, 0, 0);
\coordinate (X) at (1,0,-1);
\coordinate (Y) at (0, -1, 1);
\coordinate (Z) at (-1, 1, 0); 

\VLozenge{($-3 *(X)$)}
\ULozenge{(Z)}
\WLozenge{($2 *(X)$)}

\draw[red] ($-3 *(X)$) -- ($-3 *(X) - (Y)$);
\draw[red] (Z) -- ($(Z) + (X)$);
\draw[red] ($3 *(X)$) -- ($3 *(X) + (Z)$);

\node at (Z) {$\bullet$};
\node at ($-3 *(X)$) {$\bullet$};
\node at ($2 *(X)$) {$\bullet$};

\filldraw[white] ($-2.5 *(X) + 0.5 *(Z)$) circle (8pt);
\filldraw[white] ($0.5 *(Z) - 0.5 *(Y)$) circle (8pt);
\filldraw[white] ($2.5 *(X) - 0.5 *(Y)$) circle (8pt);

\node at ($-2.5 *(X) + 0.5 *(Z) $) {$V$};
\node at ($0.5 *(Z) - 0.5 *(Y)$) {$U$};
\node at ($2.5 *(X) - 0.5 *(Y)$) {$W$};

\node at ($(Z)-0.15 *(X)$) {$x$};
\node at ($-3.15 *(X)$) {$x$};
\node at ($1.85 *(X)$) {$x$};

\end{scriptsize}
\end{scope}
\end{tikzpicture}
\]

Clearly, $T$ has to satisfy additional conditions to describe a tiling, for instance, the placed lozenges should not overlap or leave untiled areas. If we want our tiling to be $L_1$-periodic, further restrictions need to be imposed:

\begin{Def}\label{Def: Periodic Tiling}
    Let $L_1\subseteq L_0$ be as above. We define a $L_1$-periodic tiling as a map $T\colon L_0 \to \{U, V, W\}$ satisfying the following two axioms:
\begin{enumerate}
\item Periodicity: for all $x\in L_0$ and $y\in L_1$ we have $T(x)=T(x+y)$;
\item Compatibility: for all $ x \in L_0$, exactly one of the following conditions is satisfied:
$$T (x) = W, \quad T (x+u) = V, \quad T (x-v) = U.$$
\end{enumerate}

\end{Def}

It is easy to see in the figure below that if two or three conditions in the second axiom hold simultaneously, then the downwards pointing triangle $x-v\to x-v+u \to x+u\to x-v$ will be matched with two or three upwards pointing triangles, in other words, the placed lozenges will overlap. Conversely, if none of these conditions hold, this triangle will stay unmatched, in other words, the placed lozenges will leave an untiled area.

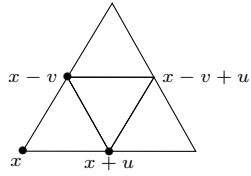
\[
\begin{tikzpicture}[-,scale=0.8]
\begin{scope}[rotate=-15.6,inner sep=1.5mm]
\begin{scriptsize}


\coordinate (O) at (0, 0, 0);
\coordinate (X) at (1,0,-1);
\coordinate (Y) at (0, -1, 1);
\coordinate (Z) at (-1, 1, 0); 

\VLozenge{(O)}
\ULozenge{($(Z)$)}
\WLozenge{($-1 *(X)$)}

\node at (O) {$\bullet$};
\node at ($(Z)$) {$\bullet$};
\node at ($-1 *(X)$) {$\bullet$};

\node at ($0.15 *(Y)- (X)$) {$x$};
\node at ($0.08 *(Y) - 0.08 *(Z)$) {$x + u$};
\node at ($-0.4 *(X)  + (Z)$) {$x-v$};
\node at ($1.6 *(X)  + (Z)$) {$x-v + u$};

\end{scriptsize}
\end{scope}
\end{tikzpicture}
\]

\begin{Rem}\label{Rem: Tilings and Periodic cuts}
    There is a bijection between $L_1$-periodic cuts on $\hat{Q}$ and $L_1$-periodic tilings. Given a periodic cut $\hat{C} \subseteq \hat{Q}_1$, the corresponding tiling is given by 
    \[ T_{\hat{C}} \colon x \mapsto \begin{cases}
        U & \textrm{ if } (x \to x+u) \in \hat{C} \\
        V & \textrm{ if } (x -v \to x) \in \hat{C} \\
        W & \textrm{ if } (x +u \to x-v) \in \hat{C}, 
    \end{cases} \]
    and given a tiling $T$ we obtain the corresponding cut $\hat{C}_T$ as the set of arrows which do not belong to the boundary of a tile in $T$. Combining this with \Cref{Pro: Periodic cuts and HPPC} we obtain a bijection between $L_1$-periodic tilings in which every type of lozenge occurs, and higher preprojective cuts. 
\end{Rem}

We will say that an arrow is in $T$ if it belongs to the boundary of a tile in $T$. Similarly, we say that a path is in $T$ if all its arrows are in $T$. Through the above reinterpretation, we give the analog of \Cref{Def: Type of a Cut}. Note that if $T$ is a $L_1$-periodic tiling, it is enough to define it on a representative of each class modulo $L_1$ and to consider the induced map $\tilde{T}\colon L_0/L_1 \to \{U, V, W\}$. We will freely switch perspectives between $L_1$-periodic tilings $T$ on $L_0$ and their corresponding maps $\Tilde{T}$ on $L_0/L_1$ when no confusion is possible.

\begin{Def}
    We define the type of a $L_1$-periodic tiling as
\[ \theta(T)=(\theta_1(T),\theta_2(T),\theta_3(T))=(\#\tilde{T}^{-1}(U),\#\tilde{T}^{-1}(V),\#\tilde{T}^{-1}(W)).\] 
Note that this definition only makes sense if $L_1$ is of rank $2$, so that $L_0/L_1$ is finite.
\end{Def}

It will be convenient to generalise the AIR-construction from \Cref{Theo:AIR cuts} to this setup. 

\begin{Rem}\label{Rem: AIR cuts for possibly zero types}
    Suppose that $L_0/L_1$ is cyclic of order $n$, and choose an isomorphism $\varphi \colon L_0/L_1 \to \mathbb{Z}/n\mathbb{Z}$. Then the generators $u$, $v$ and $w$ act on $\mathbb{Z}/n\mathbb{Z}$ by addition inside of $\mathbb{Z}/n\mathbb{Z}$. Let $e_1 \in [0, n)$ be the representative of $\varphi(u + L_1)$, and similarly let $e_2$ and $e_3$ be the standard representatives of $\varphi(v + L_1)$ and $\varphi(w + L_1)$ respectively. Then $\mathbb{Z}/n\mathbb{Z} \simeq \langle \frac{1}{n}(e_1, e_2, e_3) \rangle$, and the quiver $Q$ associated to $L_0/L_1$ is the McKay quiver for $\langle \frac{1}{n}(e_1, e_2, e_3) \rangle$. Assume that $e_1 + e_2 + e_3 = n$. It then follows that an AIR-cut on $Q$ gives an $L_1$-periodic $T$, which is is given by 
    \[ T(x) = \begin{cases}
        V, & \varphi(x-v+ L_1) > \varphi(x+ L_1) \\
        W, & \varphi(x+u+ L_1) > \varphi(x-v+ L_1) \\
        U, & \varphi(x + L_1) > \varphi(x+u+ L_1).
        \end{cases} \]
    Furthermore, even if some of the $e_i$ are zero, we can apply the same formula as for the AIR-cut to obtain an $L_1$-periodic tiling, but it no longer corresponds to a higher preprojective cut. Note that we exclude the case $n=1$ for trivial reasons. Direct counting shows that these tilings are of type $(e_1, e_2, e_3)$: Observe that there are precisely $e_1$ many arrows of type $u$ starting at $x + L_1$ with $\varphi(x + L_1) \in \{ n-e_1, n-e_1 + 1 , \ldots, n-1\}$, which is equivalent to $ \varphi(x + u +L_1) =  \varphi(x + L_1) + e_1 < \varphi(x + L_1)$. The same count applies to arrows of other types. 
\end{Rem}

Next, we introduce height functions, which will be useful throughout. 

\begin{Con}\label{Con: Height function}
Let $T$ be a $L_1$-periodic tiling. Then each vertex of $L_0$ can be given a height value as follows. We set $h_T(0)=0$. For $\alpha \in \{u,v,w \}$, we then define 
\[ h_T(x+\alpha) = \begin{cases}
    h_T(x)+1, & (x \to x+\alpha) \textrm{ is in } T \\
    h_T(x) - 2 & \textrm{ else}.
\end{cases}  \] 
It is known that this gives a well defined height function on all the vertices of $L_0$. Note that this does not descend to a single-valued height function on $L_0/L_1$. One may also notice that any lozenge tiling can be viewed as (a projection of) a pile of cubes. Any visible point in a pile of cubes can be assigned three-dimensional standard coordinates, and then $h_T(x)$ simply counts the sum of these coordinates. For a more detailed overview on piles of cubes and height functions, the reader can consult \cite{OverviewDomandLoz}.
\end{Con}

\begin{Exp}\label{tilingex}
 Let $L_1$ be the lattice defined by $B=(\begin{smallmatrix} 6 & 4 \\ 0 & 2 \end{smallmatrix})$, as in \Cref{Exp: size 12 example with simplex}. The picture below shows a $L_1$-periodic tiling $T$ and a few values of the height function associated to this tiling. The circled nodes indicate points of $L_1$, in particular, the upper left corner is the origin. This tiling has type $\theta(T)=(2,2,8)$. In order to visualise $T$ as an infinite pile of cubes and to see tilings of other types, the reader can consult \Cref{alltypes}
\[
\begin{tikzpicture}[-,scale=0.5]
\begin{scope}[rotate=-15.6,inner sep=1.5mm]
\begin{scriptsize}


\coordinate (O) at (0, 0, 0);
\coordinate (X) at (1,0,-1);
\coordinate (Y) at (0, -1, 1);
\coordinate (Z) at (-1, 1, 0);

\newcommand\U{2};
\newcommand\V{2};
\newcommand\W{8};
\pgfmathsetmacro\result{int(\V+\W)}

\foreach \x in {0,...,5}{
\foreach \y in {1,...,6}{

\pgfmathparse{int(Mod(\U*\x + \V*\y,\U+\V+\W))}

\ifthenelse{\pgfmathresult < \V}{\VLozenge{($\x *(X)+\y *(Y)$)};}{}
\ifthenelse{\pgfmathresult > \V \OR \pgfmathresult = \V \AND \pgfmathresult < \result}{\WLozenge{($\x *(X)+\y *(Y)$)};}{}
\ifthenelse{\pgfmathresult > \result \OR \pgfmathresult = \result}{\ULozenge{($\x *(X)+\y *(Y)$)};}{}

}
}
\ULozenge{($5 *(X)+0 *(Y)$)}

\foreach \x in {0,...,6}{
\foreach \y in {0,...,6}{
\filldraw[white] ($\x *(X)+\y *(Y)$) circle (8pt);
}
}
\node at ($0*(X)+ 0 *(Y)$) {\circled{$0$}}; 
\node at ($1*(X)+ 0 *(Y)$) {$1$}; 
\node at ($2*(X)+ 0 *(Y)$) {$2$};
\node at ($3*(X)+ 0 *(Y)$) {$3$};
\node at ($4*(X)+ 0 *(Y)$) {$4$};
\node at ($5*(X)+ 0 *(Y)$) {$5$};
\node at ($6*(X)+ 0 *(Y)$) {\circled{$3$}};

\node at ($0*(X)+ 1 *(Y)$) {$1$};
\node at ($1*(X)+ 1 *(Y)$) {$2$};
\node at ($2*(X)+ 1 *(Y)$) {$3$};
\node at ($3*(X)+ 1 *(Y)$) {$4$};
\node at ($4*(X)+ 1 *(Y)$) {$5$};
\node at ($5*(X)+ 1 *(Y)$) {$3$};
\node at ($6*(X)+ 1 *(Y)$) {$4$};

\node at ($0*(X)+ 2 *(Y)$) {$2$};
\node at ($1*(X)+ 2 *(Y)$) {$3$};
\node at ($2*(X)+ 2 *(Y)$) {$4$};
\node at ($3*(X)+ 2 *(Y)$) {$5$};
\node at ($4*(X)+ 2 *(Y)$) {\circled{$3$}};
\node at ($5*(X)+ 2 *(Y)$) {$4$};
\node at ($6*(X)+ 2 *(Y)$) {$5$};

\node at ($0*(X)+ 3 *(Y)$) {$3$};
\node at ($1*(X)+ 3 *(Y)$) {$4$};
\node at ($2*(X)+ 3 *(Y)$) {$5$};
\node at ($3*(X)+ 3 *(Y)$) {$3$};
\node at ($4*(X)+ 3 *(Y)$) {$4$};
\node at ($5*(X)+ 3 *(Y)$) {$5$};
\node at ($6*(X)+ 3 *(Y)$) {$6$};

\node at ($0*(X)+ 4 *(Y)$) {$4$};
\node at ($1*(X)+ 4 *(Y)$) {$5$};
\node at ($2*(X)+ 4 *(Y)$) {\circled{$3$}};
\node at ($3*(X)+ 4 *(Y)$) {$4$};
\node at ($4*(X)+ 4 *(Y)$) {$5$};
\node at ($5*(X)+ 4 *(Y)$) {$6$};
\node at ($6*(X)+ 4 *(Y)$) {$7$};

\node at ($0*(X)+ 5 *(Y)$) {$5$};
\node at ($1*(X)+ 5 *(Y)$) {$3$};
\node at ($2*(X)+ 5 *(Y)$) {$4$};
\node at ($3*(X)+ 5 *(Y)$) {$5$};
\node at ($4*(X)+ 5 *(Y)$) {$6$};
\node at ($5*(X)+ 5 *(Y)$) {$7$};
\node at ($6*(X)+ 5 *(Y)$) {$8$};

\node at ($0*(X)+ 6 *(Y)$) {\circled{$3$}};
\node at ($1*(X)+ 6 *(Y)$) {$4$};
\node at ($2*(X)+ 6 *(Y)$) {$5$};
\node at ($3*(X)+ 6 *(Y)$) {$6$};
\node at ($4*(X)+ 6 *(Y)$) {$7$};
\node at ($5*(X)+ 6 *(Y)$) {$8$};
\node at ($6*(X)+ 6 *(Y)$) {\circled{$6$}};









\end{scriptsize}
\end{scope}
\end{tikzpicture}
\]
 
\end{Exp}

\begin{Rem}
\label{heightadd}
 As mentioned in \Cref{Con: Height function}, the height function is not periodic and in general not additive. However, for any $x \in L_0$ and any $y\in L_1$ we have 
$h_T(x+y)=h_T(x)+h_T(y)$ and $h_T(-y)=-h_T(y)$, which in particular implies that $h_T \colon L_1 \to \mathbb{Z}$ is a group morphism, hence $h_T(ay)=ah_T(y)$ for any $a\in \mathbb{Z}$.   
\end{Rem}

\begin{Lem}
\label{heightvalues}
Let $T$ be an $L_1$-periodic tiling of type $(\gamma_1,\gamma_2,\gamma_3)$. Then for any $y=y_1u+y_2v\in L_1$ we have $h_T(y)=y_1+y_2-3\frac{\gamma_1y_1+\gamma_2y_2}{n}$
\end{Lem}
\begin{proof}
 Let $o_1$ and $o_2$ denote the orders of $u$ and $v$ respectively in $L_0/L_1$. In the first step of the proof we will show this equality for $y=o_1u$ and $y=o_2v$, and then handle the general case in the second step.

 Step 1: We would like to show that $h_T(o_1u)=o_1-3\frac{\gamma_1o_1}{n}$, the equality for $o_2v$ holds verbatim.
 Consider the following path in $L_0$ (which is not necessarily a path in $T$): $x\to x+u\to x+2u\to\ldots\to x+o_1u$. Let us compute the height increment between the two endpoints of this path. 
 Every arrow $x+iu\to x+(i+1)u$ gives an increment of $+1$ if this arrow is in $T$, and an increment of $-2$ if it is not in $T$, or, equivalently, if $T(x+iu)=U$. Combining this with \Cref{heightadd}, we obtain $h_T(o_1u)=h_T(x+o_1u)-h_T(x)=o_1-3\theta'_1(x)$, where $\theta'_1(x)=\#\{i\in[0, o_1)\mid T(x+iu)=U\}$ can be seen as the number of lozenges of type $U$ we cut as we travel along this path. From the previous equation we see that $\theta'_1(x)$ is in fact independent of $x$ and thus we may refer to it as $\theta'_1$. Therefore, in order to complete this step, it remains to show that $\theta'_1=\frac{\gamma_1o_1}{n}$. Let $\{x_1,x_2,\ldots,x_{\frac{n}{o_1}}\}$ be a set of representatives of the cosets of $\langle u \rangle$ in $ L_0/ L_1 $. 
Then we have $\gamma_1=\theta'_1(x_1)+\theta'_1(x_2)+\ldots + \theta'_1(x_{\frac{n}{o_1}})=\frac{n}{o_1}\theta'_1$, which implies the desired equality.

Step 2: Now we want to show the equality in the general case. Let $y=y_1u+y_2v$. Using \Cref{heightadd} and keeping in mind that $o_1 \mid n$ and $o_2 \mid n$, we get:

\begin{align*}
 h_T(n(y_1u+y_2v)) &=h_T(ny_1u+ny_2v)\\&=h_T(ny_1u)+h_T(ny_2v)\\&=\frac{ny_1}{o_1}h_T(o_1u)+\frac{ny_2}{o_2}h_T(o_2v)\\&=\frac{ny_1}{o_1}\left(o_1-\frac{3\gamma_1 o_1}{n}\right)+\frac{na_2}{o_2}\left(o_2-\frac{3\gamma_2 o_2}{n}\right)\\ &=ny_1+ny_2-3(\gamma_1y_1+\gamma_2y_2).  
\end{align*}
Since $h_T(ny)=nh_T(y)$, we get the desired identity. 
\end{proof}

We will now proceed to classify all triples of integers which appear as types of $L_1$-periodic tilings.

\begin{Pro}
Let $B$, $L_1$ and $n$ be as above. Let $T$ be a $L_1$-periodic tiling with $\theta(T)=(\gamma_1,\gamma_2,\gamma_3)$. Then $\gamma_1,\gamma_2,\gamma_3\ge 0$, $\gamma_1+\gamma_2+\gamma_3=n$ and \[ (\gamma_1, \gamma_2) B \in n \mathbb{Z}^{1 \times 2}. \]
\end{Pro}

\begin{proof}
It is clear that $\gamma_1,\gamma_2,\gamma_3\ge 0$ and $\gamma_1+\gamma_2+\gamma_3=n$. Recall from \Cref{Rem: Positive matrix} that we may choose $B$ with non-negative entries. It suffices to show the claim for this choice of $B$ since clearly the $\SL_2(\mathbb{Z})$-action leaves invariant the sublattice $n\mathbb{Z}^{1\times 2}$, i.e. $ n \mathbb{Z}^{1 \times 2} \cdot \SL_2(\mathbb{Z}) = n \mathbb{Z}^{1\times 2}$. 
In order to prove the proposition, we will show that $\gamma_1c_1+\gamma_2c_2\in n\mathbb{Z}$ for any $c_1,c_2\ge 0$ such that $c_1u+c_2v\in L_1$. This will imply the desired result. 

Consider the following path in $L_0$ (which is not necessarily a path in $T$): $0\to u\to 2u\to\ldots\to c_1u\to c_1u+v\to c_1u+2v\to\ldots\to c_1u+c_2v$. The path from $0$ to $c_1u$ gives a height increment of $c_1-3\varepsilon_1$, where $\varepsilon_1$ is the number of lozenges of type $U$ we cut on the way. Similarly, the path from $c_1u$ to $c_1u+c_2v$ gives an increment of $c_2-3\varepsilon_2$, where $\varepsilon_2$ is the number of lozenges of type $V$ we cut on the way. The total height increment is then $h_T(c_1u+c_2v)=c_1+c_2-3\varepsilon_1-3\varepsilon_2$. On the other hand, we already know from \Cref{heightvalues} that $h_T(c_1u+c_2v)=c_1+c_2-3\frac{\gamma_1c_1+\gamma_2c_2}{n}$. Comparing the two expressions, we obtain $\gamma_1c_1+\gamma_2c_2=n(\varepsilon_1+\varepsilon_2)$, that is, $n \mid (\gamma_1c_1+\gamma_2c_2)$.
\end{proof}    

To prove the existence of tilings of a given type, we use the following construction. 

\begin{Lem}\label{Lem: xi is a hom}
   Let $B$, $L_1$ and $n$ be as above. Let $\gamma = (\gamma_1,\gamma_2,\gamma_3)$ be a triple of integers such that $\gamma_1,\gamma_2,\gamma_3\ge 0$, $\gamma_1+\gamma_2+\gamma_3=n$ and $(\gamma_1, \gamma_2) B \in n \mathbb{Z}^{1 \times 2}$.
   Then 
   \[ \xi_\gamma \colon L_0 \to \mathbb{Z}/n\mathbb{Z}, (x_1u + x_2v) \to (x_1\gamma_1+x_2\gamma_2) \bmod n  \]
   is a group homomorphism with $ \Ker(\xi_\gamma) \supseteq L_1$ and $ \Im(\xi_\gamma)$ is cyclic of order $\frac{n}{\gcd(\gamma_1, \gamma_2, \gamma_3)}$.
\end{Lem}

\begin{proof}
    It is clear that $\xi_\gamma$ is a group homomorphism. We first investigate 
    \[ \Ker(\xi_\gamma) =\{x_1u+x_2v\in L_0\mid  (x_1\gamma_1+x_2\gamma_2) \equiv_n 0 \}. \]
    Recall that any $y \in L_1$ can be written as $y = B  (y_1, y_2)^\top$ for some $(y_1, y_2)^\top \in \mathbb{Z}^2$. Hence we have   
    \[ \xi_\gamma(y) = (\gamma_1, \gamma_2) B  (y_1, y_2)^\top \equiv_n 0,   \]
    and therefore $\Ker(\xi_\gamma) \supseteq L_1$. To compute the order of $\Im(\xi_\gamma)$, note that we can write $d = \gcd(\gamma_1, \gamma_2, \gamma_3) = \gcd(\gamma_1, \gamma_2, n)$ as an integer combination of $\gamma_1$, $\gamma_2$ and $n$. That means there exist numbers $x_1, x_2 \in \mathbb{Z}$ such that $\xi_\gamma(x_1u + x_2 v) \equiv_n d $. Since $d \mid \gamma_1$ and $d \mid \gamma_2$, it follows that the subgroup $\Im(\xi_\gamma) \subseteq  \mathbb{Z}/n\mathbb{Z}$ is generated by the coset of $d$, and therefore has order $\frac{n}{\gcd(\gamma_1, \gamma_2, \gamma_3)}$.
\end{proof}

\begin{Pro}
\label{tilingconstr}
Let $B$, $L_1$, $n$, $d$ and $(\gamma_1, \gamma_2, \gamma_3)$ be as in \Cref{Lem: xi is a hom} and let $L_2 = \Ker(\xi_\gamma)$. Then there exists an $L_2$-periodic tiling $T_\gamma$ of type $(\frac{\gamma_1}{d},\frac{\gamma_2}{d}, \frac{\gamma_3}{d} )$. In particular, $T_\gamma$ is $L_1$-periodic of type $(\gamma_1, \gamma_2, \gamma_3)$. 
\end{Pro}

\begin{proof}
    We write $n' = \frac{n}{d}$. If $n' = 1$, one of the $\gamma_i$ is $n$ and $L_2 = L_0$, hence we take $T_\gamma$ to be the constant tiling of the corresponding type. Otherwise, by \Cref{Lem: xi is a hom}, the morphism $\xi_\gamma$ has a cyclic image $\Im(\xi_\gamma) \simeq \mathbb{Z}/n'\mathbb{Z}$. Hence, we have that $L_0/L_2$ is cyclic of order $n'$, and we will use the AIR-cut to construct an $L_2$-periodic tiling, see \Cref{Rem: AIR cuts for possibly zero types}. To do this, we factor $\xi_\gamma$ as the projection $q \colon L_0 \to L_0/L_2$, followed by the induced isomorphism $\overline{\xi_\gamma}$ to $\mathbb{Z}/n'\mathbb{Z}$, and the inclusion of $\mathbb{Z}/n'\mathbb{Z}$ as $\Im(\xi_\gamma) \leq \mathbb{Z}/n\mathbb{Z}$. 
    \[\begin{tikzcd}
	{L_0} & {\mathbb{Z}/n\mathbb{Z}} \\
	{L_0/L_2} & {\mathbb{Z}/n'\mathbb{Z}}
	\arrow["\xi_\gamma", from=1-1, to=1-2]
	\arrow["q"', from=1-1, to=2-1]
	\arrow["{\overline{\xi_\gamma}}"', "\sim", from=2-1, to=2-2]
	\arrow["{\cdot d}"', from=2-2, to=1-2]
    \end{tikzcd}\]
    Let us compute the representation of $\mathbb{Z}/n'\mathbb{Z}$ given by the isomorphism $\overline{\xi_\gamma}$. We have that $\xi_\gamma(u) = \xi_\gamma(1 \cdot u + 0 \cdot v) = \gamma_1$, and similarly $\xi_\gamma(v) = \gamma_2$ and $\xi_\gamma(w) = \gamma_3$, and hence $\overline{\xi_\gamma}(u+L_2) = \frac{\gamma_1}{d}$. Therefore, the quiver for $L_0/L_2$ is the McKay quiver of the group $G = \langle \frac{1}{n'}(\frac{\gamma_1}{d},\frac{\gamma_2}{d},\frac{\gamma_3}{d} )\rangle $. Recall that we assumed $\gamma_1 + \gamma_2 + \gamma_3 = n$, and hence 
    $ \frac{\gamma_1}{d} + \frac{\gamma_2}{d} + \frac{\gamma_3}{d} = \frac{n}{d} = n' $. Therefore, there exists an AIR-cut for $G$, and the discussion in \Cref{Rem: AIR cuts for possibly zero types} shows that this cut gives an $L_2$-periodic tiling $T_\gamma$ of type $(\frac{\gamma_1}{d},\frac{\gamma_2}{d},\frac{\gamma_3}{d} )$. Since we assumed that $(\gamma_1, \gamma_2) B \in n \mathbb{Z}^{1 \times n}$, it follows from \Cref{Lem: xi is a hom} that $L_2 \supseteq L_1$. Therefore, the tiling $T_\gamma$ is also $L_1$-periodic, and since $L_0/L_2$ has index $d$ in $L_0/L_1$, it follows that the type of $T_\gamma$ as an $L_1$-periodic tiling is $d \cdot (\frac{\gamma_1}{d},\frac{\gamma_2}{d},\frac{\gamma_3}{d} ) = (\gamma_1,\gamma_2,\gamma_3 )$.  
\end{proof}

We arrive to the following result.
\begin{Theo}
\label{Theo: Divisibility conditions for tilings}

 Let $B = (\begin{smallmatrix} a_1 & b_1 \\ a_2 & b_2 \end{smallmatrix})$ be a $2$-rank integer matrix as in \Cref{Pro: Periodicity matrix}, and let $L_1$ be the lattice defined by $B$. Let $n= \det B =|L_0/L_1|$. Then $(\gamma_1,\gamma_2,\gamma_3)$ is the type of some $L_1$-periodic tiling if and only if $\gamma_1,\gamma_2,\gamma_3\ge 0$, $\gamma_1+\gamma_2+\gamma_3=n$ and \[ (\gamma_1, \gamma_2) B \in n \mathbb{Z}^{1 \times 2}. \] 
\end{Theo}

\begin{Rem}
 We would like to point out that \Cref{Theo: Divisibility conditions for tilings} is a restatement of \cite[Theorem 6]{Martin}. In order to see this, one can explicitly construct the affine bijection sending the fundamental triangle of possible fingerprints from \cite{Martin} to the triangle of potential types, that is, $\langle(n,0,0),(0,n,0),(0,0,n)\rangle$. Then the condition from \cite{Martin} translates to the condition in \Cref{Theo: Divisibility conditions for tilings}. 
\end{Rem}

Restricting the above theorem to those types which correspond to higher preprojective cuts provides the proof of \Cref{Theo: Divisibility conditions}. 

\begin{proof}[Proof of \Cref{Theo: Divisibility conditions}]
    Translate \Cref{Theo: Divisibility conditions for tilings} via \Cref{Rem: Tilings and Periodic cuts}. 
\end{proof}

Next, we give an explicit formula for the tiling $T_\gamma$. 

\begin{Pro}\label{Pro: General xi-formula for T}
Let $T_\gamma$ be the tiling constructed in \Cref{tilingconstr}. Then $T_\gamma$ is given by
\[    T_\gamma(x)= 
\begin{cases}
    V,& \text{if } \xi_{\gamma}(x) \in [0,\gamma_2),\\
    W,& \text{if } \xi_{\gamma}(x) \in [\gamma_2,\gamma_2+\gamma_3),\\
   U,& \text{if } \xi_{\gamma}(x) \in [\gamma_2+\gamma_3,n).
    
    \end{cases}
\]
\end{Pro}

\begin{proof}
If $T_\gamma$ is a constant tiling, the above formula holds trivially. Otherwise, let $\xi_{\gamma}(x)\in [\gamma_2,\gamma_2+\gamma_3)$. Since $\xi_{\gamma} (x) =  d \cdot \overline{\xi_\gamma} (q(x))$, this is equivalent to $\overline{\xi_\gamma}(x+L_2)\in [\frac{\gamma_2}{d},\frac{\gamma_2+\gamma_3}{d})$. This is the same as $\overline{\xi_\gamma}(x+L_2)>\overline{\xi_\gamma}(x+u+L_2)$, which, by \Cref{Rem: AIR cuts for possibly zero types}, is equivalent to the arrow $x\rightarrow x+u$ being removed, in other words, $T_{\gamma}(x)=U$. One can similarly consider the other two intervals.
\end{proof}

\begin{Exp}\label{alltypes}
Let $L_1$ be the lattice from \Cref{Exp: size 12 example with simplex}, i.e. $L_1$ is defined by $B=(\begin{smallmatrix} 6 & 4 \\ 0 & 2 \end{smallmatrix})$. There, we used \Cref{Theo: Divisibility conditions for tilings} to conclude that the only possible types of $L_1$-periodic tilings are: $(12,0,0)$, $(0,12,0)$, $(0,0,12)$, $(0,6,6)$, $(6,0,6)$, $(6,6,0)$, $(8,2,2)$, $(2,8,2)$, $(2,2,8)$, $(4,4,4)$. For each of these types, a tiling of the corresponding type is drawn in the picture below. Note that in general there exist multiple tilings of a given type. The tilings given below are those constructed in \Cref{tilingconstr} and have been computed using \Cref{Pro: General xi-formula for T}, and they are arranged according to the positioning of their types in the simplex in \Cref{Exp: size 12 example with simplex}.
 
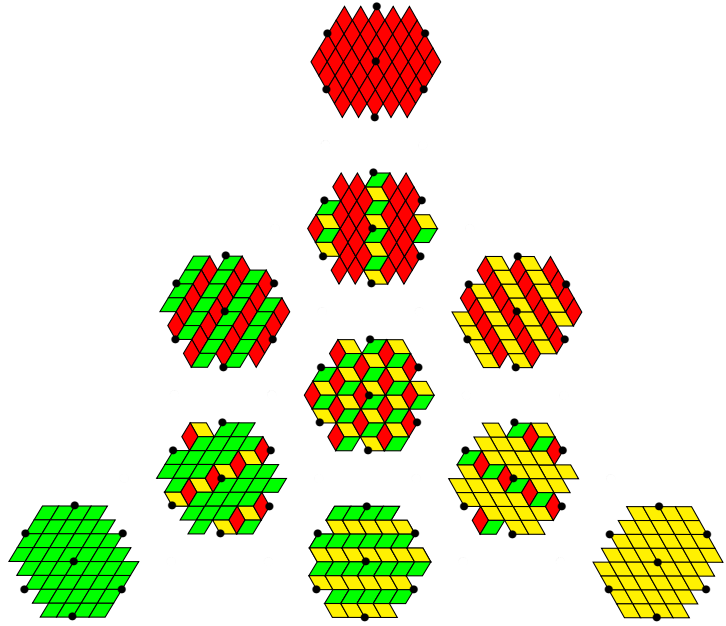
\[
\begin{tikzpicture}[-,scale=0.15]
\begin{scope}[rotate=-15.6,inner sep=1.5mm]
\begin{scriptsize}


\coordinate (O) at (0, 0, 0);
\coordinate (X) at (1,0,-1);
\coordinate (Y) at (0, -1, 1);
\coordinate (Z) at (-1, 1, 0);

\foreach \y in {36, 34}{
    \pgfmathsetmacro\lshift{int((36 - \y) + 15) }
    \pgfmathsetmacro\rshift{int(21 - (36 - \y)) }
    \foreach \x in {15, ..., \rshift}{
    \fVLozenge{($\x *(X)+\y *(Y)$)};
    }

}
\foreach \y in {35, 33}{
    \pgfmathsetmacro\lshift{int((36 - \y) + 15) }
    \pgfmathsetmacro\rshift{int(19 - (35 - \y)) }
    \foreach \x in {14, ..., \rshift}{
    \fWLozenge{($\x *(X)+\y *(Y)$)};
    }
}

\foreach \x in {-4, -2, ..., 2}{
\ifthenelse{\x < 0}
{ \pgfmathtruncatemacro\uplim{22 + \x} 
    \foreach \y in {15, ..., \uplim}
    {
        \fWLozenge{($ \x *(X) + \y *(Y) $)};
        \fULozenge{($ \x  *(X) + (X) + \y *(Y) $)};
        \fULozenge{($ \x  *(X) + (X) + \y *(Y) + (Y) $)};
    }
}
{   \pgfmathtruncatemacro\lowlim{\x + 16}
    \foreach \y in {\lowlim, ..., 21}
    {
        \fWLozenge{($ \x *(X) + \y *(Y) $)};
        \fWLozenge{($ \x *(X) + \y *(Y) + (Y)$)};
        \fULozenge{($ \x  *(X) + (X) + \y *(Y) $)};
    }
}
}


\foreach \y in {33, 35, ..., 39}{
\ifthenelse{\y < 36}
{ \pgfmathtruncatemacro\rlim{\y - 16 } 
    \foreach \x in {14, ..., \rlim}
    {
        \fWLozenge{($ \x *(X) + \y *(Y) $)};
        \fVLozenge{($ \x  *(X) + (X) + \y *(Y) + (Y) $)};
        \fVLozenge{($ \x  *(X) + (X) + (X) + \y *(Y) + (Y) $)};
    }
}
{   \pgfmathtruncatemacro\llim{\y - 22}
    \foreach \x in {\llim, ..., 20}
    {
        \fWLozenge{($ \x *(X) + \y *(Y) $)};
        \fVLozenge{($ \x  *(X) + (X) + \y *(Y) + (Y) $)};
        \fWLozenge{($ \x *(X) + (X) + \y *(Y) $)}
    }
}
}

 
\foreach \x in {-3, -1, ..., 3}{
\ifthenelse{\x < 0}
{ \pgfmathtruncatemacro\llim{15 - \x } 
    \foreach \z in {\llim, ..., 21}
    {
        \fULozenge{($ \x *(X) - \z *(Z) $)};
        \fULozenge{($\x *(X) - \z *(Z) + (Z) $)};
        \fVLozenge{($ \x  *(X) - \z *(Z) + (Y) $)};
    }
}
{   \pgfmathtruncatemacro\ulim{21 - \x }
    \foreach \z in {15, ..., \ulim}
    {
        \fULozenge{($ \x *(X) - \z *(Z) $)};
        \fVLozenge{($ \x  *(X) - \z *(Z) + (Y) $)};
        \fVLozenge{($ \x  *(X) - \z *(Z) + (Y) + (Z) $)};
    }
}
}


\foreach \x in {-4, ..., -1}{
\pgfmathtruncatemacro\uplim{40 + \x}
\foreach \y in {33, ..., \uplim}{
\fWLozenge{($ \x *(X) + \y *(Y) $)};
}}

\foreach \x in {0, 1, ..., 3}{
\pgfmathtruncatemacro\lowlim{34 + \x}
\foreach \y in {\lowlim, ..., 40}{
\fWLozenge{($ \x *(X) + \y *(Y) $)};
}}

\foreach \x in {4,..., 8}{
\foreach \h in {-1, 0,1}{
    \fULozenge{($\x *(X) - \h *(Y) + \h *(Z) + 12 *(Y)$) };
    \fULozenge{($\x *(X) - \h *(Y) + \h *(Z) + 12 *(Y) + (Z)$) };
    \fULozenge{($\x *(X) - \h *(Y) + \h *(Z) + 12 *(Y) + (Y)$) };
}
}

\foreach \h in {-1, 0, 1}{
\fCube{($6 *(X) + 12 *(Y) + \h *(Y) - \h *(Z)  $)};
}
\fCube{($3 *(X) + 12 *(Y)$)};
\fCube{($9 *(X) + 12 *(Y)$)};
\fWLozenge{($3 *(X) + 12 *(Y) + (Z)$)};
\fWLozenge{($ 4 *(X) + 9 *(Y)$)};

\fVLozenge{($3 *(X) + 13 *(Y) - (Z)$)};
\fVLozenge{($8 *(X) + 16 *(Y)$)};

\foreach \z in {3, ..., 7}{
\foreach \h in {-1, 0, 1}{
\fWLozenge{($ 25 *(Y)  - \z *(Z) + \h *(Y) - \h *(X) $)}; 
\fWLozenge{($ 25 *(Y)  - \z *(Z) + \h *(Y) - \h *(X) + (X)$)};
\fWLozenge{($ 25 *(Y)  - \z *(Z) + \h *(Y) - \h *(X) + (Y) $)};
}
}

\fCube{($3 *(X)  +27 *(Y)$)};
\fCube{($3 *(X)  +27 *(Y) - 6 *(Z)$)};
\foreach \h in {-1, 0, 1}{
\fCube{($ 6 *(X) + 30 *(Y) + \h *(X) - \h *(Y) $)};
}

\fVLozenge{($ 4 *(X) + 32 *(Y)$)};
\fVLozenge{($ 6 *(X) + 32 *(Y) - 2 *(Z)$)};
\fULozenge{($ 7 *(X)  + 28 *(Y) $)};
\fULozenge{($ 5 *(X)  + 28 *(Y) - 4 *(Z) $)};

\foreach \y in {29, ..., 33}{
\foreach \h in {-1, 0, 1}{
\fVLozenge{($ \y *(Y) + 24 *(X) + \h *(Z) - \h *(X) $)};
\fVLozenge{($ \y *(Y) + 24 *(X) + \h *(Z) - \h *(X) + (Z)$)};
\fVLozenge{($ \y *(Y) + 24 *(X) + \h *(Z) - \h *(X) + (X)$)};
}
}

\foreach \h in {-1, 0, 1}{
\fCube{ ($ 24 *(X) + 30 *(Y) + \h *(X) - \h *(Z)$)};
}

\fCube{($ 24 *(X) + 33 *(Y)$)};
\fCube{($ 24 *(X) + 27 *(Y)$)};

\fWLozenge{($ 20 *(X) + 29 *(Y)$)}
\fWLozenge{($ 22 *(X) + 27 *(Y)$)}

\fULozenge{($ 25 *(X) + 28 *(Y)$)}
\fULozenge{($ 27 *(X) + 32 *(Y)$)}


\foreach \x in {33, ..., 36}{
\pgfmathtruncatemacro\uplim{4 + \x}
\foreach \y in {33, ..., \uplim}{
\fVLozenge{($ \x *(X) + \y *(Y) $)};
}}

\foreach \x in {37, ..., 39}{
\pgfmathtruncatemacro\lowlim{ \x- 3}
\foreach \y in {\lowlim, ..., 40}{
\fVLozenge{($ \x *(X) + \y *(Y) $)};
}}

\foreach \x in {-4, ..., -1}{
\pgfmathtruncatemacro\uplim{4 + \x}
\foreach \y in {-3, ..., \uplim}{
\fULozenge{($ \x *(X) + \y *(Y) $)};
}}

\foreach \x in {0, ..., 3}{
\pgfmathtruncatemacro\lowlim{ \x- 3}
\foreach \y in {\lowlim, ..., 3}{
\fULozenge{($ \x *(X) + \y *(Y) $)};
}}

\foreach \x in {9, ..., 12}{
\pgfmathtruncatemacro\uplim{ \x + 15 }
\foreach \y in {21, ..., \uplim}{

\pgfmathtruncatemacro\ok{ int(Mod(\x + \y, 3))  }
\ifthenelse{\ok = 0}{
\fCube{($\x *(X) + \y *(Y)$) };}{}
}
}
\foreach \x in {13, ..., 15}{
\pgfmathtruncatemacro\lowlim{ \x + 9 }
\foreach \y in {\lowlim, ..., 27}{
\pgfmathtruncatemacro\ok{ int(Mod(\x + \y, 3))  }
\ifthenelse{\ok = 0}{
\fCube{($\x *(X) + \y *(Y)$) };}{}
}
}

\fULozenge{($ 13 *(X) + 22 *(Y) $)};
\fULozenge{($ 15 *(X) + 26 *(Y) $)};
\fVLozenge{($ 14 *(X) + 28 *(Y) $)};
\fVLozenge{($ 10 *(X) + 26 *(Y) $)};
\fWLozenge{($  8 *(X) + 23 *(Y) $)};
\fWLozenge{($ 10 *(X) + 21 *(Y) $)};

\foreach \x in {0,6,...,30}{
\pgfmathtruncatemacro\shift{\x + 6}
\foreach \y in {\x, \shift,...,36}{
\node at ($\x *(X) + \y *(Y)$) {$\bullet$};
}
}

\foreach \x in {2,8,...,30}{
\pgfmathtruncatemacro\shift{\x + 6}
\foreach \y in {\x, \shift,...,36}{
\node at ($\x *(X) + \y+2 *(Y)$) {$\bullet$};
}
}

\foreach \x in {4,10,...,24}{
\pgfmathtruncatemacro\shift{\x + 6}
\foreach \y in {\x, \shift,...,30}{
\node at ($\x *(X) + \y+4 *(Y)$) {$\bullet$};
}
}

\node at ($ 36 *(X) + 36 *(Y) $) {$\bullet$};
\node at ($ 32 *(X) + 34 *(Y) $) {$\bullet$};
\node at ($28 *(X) + 32 *(Y)$) {$\bullet$};

\node at ($4 *(X) + 2 *(Z) $) {$\bullet$};
\node at ($-4 *(X) - 2 *(Y) $) {$\bullet$};
\node at ($4 *(X) + 2 *(Y) $) {$\bullet$};
\node at ($-4 *(X) - 2 *(Z) $) {$\bullet$};
\node at ($2 *(X) +4 *(Z) $) {$\bullet$};

\node at ($2 *(X) + 4 *(Z) + 36 *(Y) $) {$\bullet$};
\node at ($2 *(X) + 40 *(Y) $) {$\bullet$};
\node at ($4 *(X) + 38 *(Y) $) {$\bullet$};
\node at ($-2 *(X) + 2 *(Z) + 36 *(Y)$) {$\bullet$};
\node at ($-2 *(X) + 38 *(Y) $) {$\bullet$};

\node at ($-32 *(Z) + 2 *(X) $) {$\bullet$};
\node at ($-34 *(Z) + 4 *(X) $) {$\bullet$};
\node at ($-38 *(Z) + 2 *(X) $) {$\bullet$};
\node at ($-38 *(Z) - 4 *(X) $) {$\bullet$};
\node at ($-40 *(Z) - 2 *(X) $) {$\bullet$};

\node at ($ 16 *(X) + 38 *(Y) $) {$\bullet$};
\node at ($ 22 *(X) + 38 *(Y) $) {$\bullet$};
\node at ($ 20 *(X) + 40 *(Y) $) {$\bullet$};

\node at ($ -2 *(X) + 20 *(Y) $) {$\bullet$};
\node at ($ -4 *(X) + 16 *(Y) $) {$\bullet$};
\node at ($ -2 *(X) + 14 *(Y) $) {$\bullet$};

\node at ($ 2 *(X) - 14 *(Z) $) {$\bullet$};
\node at ($ 4 *(X) - 16 *(Z) $) {$\bullet$};
\node at ($ 2 *(X) - 20 *(Z) $) {$\bullet$};

\filldraw[white] ($ 6 *(Y)$) circle (12pt);
\filldraw[white] ($6 *(X) + 6 *(Y)$) circle (12pt);
\filldraw[white] ($ 12 *(Y)$) circle (12pt);
\filldraw[white] ($12 *(X) + 12 *(Y)$) circle (12pt);
\filldraw[white] ($6 *(X) + 18 *(Y)$) circle (12pt);
\filldraw[white] ($12 *(X) + 18 *(Y)$) circle (12pt);
\filldraw[white] ($0 *(X) + 24 *(Y)$) circle (12pt);
\filldraw[white] ($6 *(X) + 24 *(Y)$) circle (12pt);
\filldraw[white] ($18 *(X) + 24 *(Y)$) circle (12pt);
\filldraw[white] ($24 *(X) + 24 *(Y)$) circle (12pt);
\filldraw[white] ($0 *(X) + 30 *(Y)$) circle (12pt);
\filldraw[white] ($12 *(X) + 30 *(Y)$) circle (12pt);
\filldraw[white] ($18 *(X) + 30 *(Y)$) circle (12pt);
\filldraw[white] ($30 *(X) + 30 *(Y)$) circle (12pt);
\filldraw[white] ($6 *(X) + 36 *(Y)$) circle (12pt);
\filldraw[white] ($12 *(X) + 36 *(Y)$) circle (12pt);
\filldraw[white] ($24 *(X) + 36 *(Y)$) circle (12pt);
\filldraw[white] ($30 *(X) + 36 *(Y)$) circle (12pt);

\end{scriptsize}
\end{scope}
\end{tikzpicture}
\]

\end{Exp}

Let us make some observations about the construction of $T_\gamma$, see also \Cref{Cor: Tilting equiv to skewed AIR} for representation-theoretic consequences. 

\begin{Rem}\label{Rem: Our tilings are lifted}
    For arbitrary $L_1 \leq L_0$, we found a lattice $L_2 \supseteq L_1$ such that $L_0/L_2$ is cyclic, and used the AIR-cut on the associated quivers to produce non-constant $L_2$-periodic tilings. The resulting $L_1$-periodic tiling then corresponds to a cut that arises in a procedure analogous to \Cref{Cor: Lifting strategy}. We therefore obtain a strengthened version of the proof of \Cref{Theo:SL3 Classification}: For every type of higher preprojective cut, we can construct one that arises as the skewed higher preprojective cut of an AIR-cut. 
\end{Rem}

Furthermore, the tiling $T_\gamma$ from \Cref{tilingconstr} obtains the maximal number of sinks and sources.

\begin{Def}
    We call a point $x$ in $L_0$ a \emph{source} resp. \emph{sink} in the tiling $T$ if it is a source or a sink in the corresponding cut quiver $\hat{Q}_{\hat{C}_T}$. Similarly, we call a point $x + L_1$ in $L_0/L_1$ a source or a sink if $x$ is one in $L_0$. 
\end{Def}

Note that a source or a sink arises in a tiling precisely when three lozenges of three different types meet in a point. Clearly, for any tiling of type $(\gamma_1,\gamma_2,\gamma_3)$ the number of sources in $L_0/L_1$ is bounded above by $\min(\gamma_1,\gamma_2,\gamma_3)$ since any source requires a lozenge of each type. The same holds for sinks. We will now show that it is possible to attain both of these maxima simultaneously:

\begin{Pro}
    Let $(\gamma_1,\gamma_2,\gamma_3)$ be as in \Cref{tilingconstr}, and let $T_{\gamma}$ be the tiling constructed in the corresponding proof. Then $T_{\gamma}$ has $\min(\gamma_1,\gamma_2,\gamma_3)$ sources and $\min(\gamma_1,\gamma_2,\gamma_3)$ sinks.
\end{Pro}

\begin{proof}
    We use the same notation as in the proof of \Cref{tilingconstr}, i.e. we set $L_2 = \Ker(\xi_\gamma)$, $d = \gcd(\gamma_1, \gamma_2, n)$, $n'=\frac{n}{d}$ and we write $\gamma_i' = \frac{\gamma_i}{d}$. We consider $T_\gamma$ as an $L_2$-periodic tiling of type $(\gamma_1', \gamma_2', \gamma_3'  )$. Recall that $\xi_\gamma$ factors so that $\xi_\gamma(x) = d \cdot \overline{\xi_\gamma}(x + L_2)$, where $\overline{\xi_\gamma} \colon L_0/L_2 \to \mathbb{Z}/n'\mathbb{Z}$ is the induced isomorphism. Then, by \Cref{Rem: AIR cuts for possibly zero types}, an arrow $x + L_2 \to x + u + L_2$ of type $u$ is not in $T_\gamma$ if and only if $\overline{\xi_\gamma}(x + L_2) > \overline{\xi_\gamma}(x + u + L_2)$, which in turn is the case if and only if $\overline{\xi_\gamma}(x) \in \{ n-\gamma_1', n-\gamma_1' + 1 , \ldots, n-1  \}$. Similarly, arrows $x + L_2 \to x+v + L_2$ and $x + L_2 \to x+w + L_2$ are not in $T_\gamma$ if and only if $\overline{\xi_\gamma}(x) \in \{ n-\gamma_2',  \ldots, n-1  \}$ and $\overline{\xi_\gamma}(x) \in \{ n-\gamma_3', \ldots, n-1  \}$, respectively. A sink is therefore precisely a vertex $x + L_2$ with $\overline{\xi_\gamma}(x) \in \{ n-\min(\gamma_1', \gamma_2', \gamma_3'), \ldots, n-1  \}$. Hence, there are precisely $\min(\gamma_1', \gamma_2', \gamma_3')$ sinks in $L_0/L_2$. Similarly, one can show that the sources are given precisely by the points $x+L_2$ with $\overline{\xi_\gamma}(x) \in \{0, 1, \ldots, \min(\gamma_1', \gamma_2', \gamma_3')-1 \}$, and hence there are also $\min(\gamma_1', \gamma_2', \gamma_3')$ sources in $T_\gamma$ when seen as an $L_2$-periodic tiling. 
    Since $L_0/L_2$ has index $d$ in $L_0/L_1$, it follows that $T_\gamma$ as an $L_1$-periodic tiling has $d \cdot \min(\gamma_1', \gamma_2', \gamma_3') = \min(\gamma_1, \gamma_2, \gamma_3)$ sinks and equally many sources. 
\end{proof}

\subsection{Mutation and flips}\label{SSec: Mutation and flips}
Recall that a source or a sink $x$ in a tiling $T$ arises when three lozenges of different types meet in the point $x$. In this case, we can produce a new tiling $\mu_x(T)$ of the same type by a procedure called a \emph{flip}. For a source, we can replace the three meeting lozenges with the corresponding configuration for a sink, and vice versa. If a flip is performed at $x \in L_0$, we perform it at all $L_1$-translates of $x$ simultaneously. 

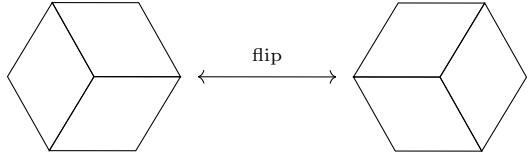
\[
\begin{tikzpicture}[-,scale=0.8]
\begin{scope}[rotate=-15.6,inner sep=1.5mm]
\begin{scriptsize}

\coordinate (O) at (0, 0, 0);
\coordinate (X) at (1,0,-1);
\coordinate (Y) at (0, -1, 1);
\coordinate (Z) at (-1, 1, 0); 

\VLozenge{(-2, 0, 2)}
\ULozenge{(-3,0,3)}
\WLozenge{(-2,-1,3)}

\WLozenge{(1,0,-1)}
\ULozenge{(2,0,-2)}
\VLozenge{(2,-1,-1)}

\draw ($-0.8 *(X) $) edge[<->, "flip"]  ($ 0.8 *(X) $);

\end{scriptsize}
\end{scope}
\end{tikzpicture}
\]

\begin{Rem}
Note that a flip of a tiling $T$ corresponds precisely to a mutation of the corresponding cut $C_T$ as in \Cref{Def: Mutation of cuts}.     
\end{Rem}

We will prove that all tilings of a given type are related by a sequence of flips. Note that the following lemma excludes the case of $0$ being a source or a sink. However, this case will not be of relevance for us.

\begin{Lem}\label{Lem: Flipping height function}
    Given a tiling $T$ of type $\theta(T) = (\gamma_1, \gamma_2, \gamma_3)$ with $\gamma_1, \gamma_2, \gamma_3 > 0$ and a source $z \neq 0$ in $T$, denote $T' = \mu_z(T)$ the flipped tiling. Then the height function of $T'$ coincides with that of $T$, except at the points $z + L_1$. More precisely, we have 
    \[ h_{T'}(x) = \begin{cases}
        h_T(x), & x \not \in z + L_1 \\ 
        h_T(x) +3, & \textrm{else}.
    \end{cases} \]
    Similarly, for a sink $z \neq 0$ in $T$ and the flipped tiling $T' = \mu_z(T)$ we have 
    \[ h_{T'}(x) = \begin{cases}
        h_T(x), & x \not \in z + L_1 \\ 
        h_T(x) -3, & \textrm{else}.
    \end{cases} \]
\end{Lem}

\begin{proof}
    We treat the case of $z$ being a source, the sink case is analogous. Recall that the height function at $x$ was defined in terms of paths in $L_0$ from $0$ to $x$. If $x \not \in z + L_1$, consider an arbitrary path $0 \to p_1 \to p_2 \to \cdots \to x$. Now suppose that this path passes through some $p_i \in z + L_1$, and that this is the first time. Since a flip only changes the tiling $T$ at lozenges which meet at $z + L_1$, this means that $h_T(p_j) = h_{T'}(p_j)$ for all $j < i$. In $T$, we have that $p_i$ is a source, while in $T'$ it is a sink, hence we see that 
    \[ h_T(p_{i+1}) = h_T(p_{i-1}) -2 + 1 = h_{T'}(p_{i-1}) + 1 - 2 = h_{T'}(p_{i+1}).   \]
    From this we see that $h_T(x) = h_{T'}(x)$ for all $x \not \in z + L_1$. 
    Now, consider the point $z-u$, for which we have $h_T(z-u) = h_{T'}(z-u)$. By definition of the height function, and since $z$ is a source in $T$, we have 
    \[ h_{T'}(z) - h_T(z)  = h_{T'}(z-u) + 1 - (h_T(z-u) - 2 ) = 3. \qedhere \]
\end{proof}

The statement of the following lemma follows from the proof of \Cref{Pro: HPPCut iff positive and homogeneous}, but we give a purely combinatorial proof using height functions for the reader's convenience. 

\begin{Lem}\label{Lem: Positive type is acyclic}
Let $T$ be a tiling of type $\theta(T) = (\gamma_1, \gamma_2, \gamma_3)$ such that $\gamma_1, \gamma_2, \gamma_3 > 0$. Then there is no infinite path in $T$. Equivalently, the cut quiver for $L_0 / L_1$ is acyclic.  
\end{Lem}

\begin{proof}
 Assume there is an infinite path in $T$. Since $L_0/L_1$ is a finite group, there will exist $x\to\ldots\to x+y$ such that $y\in L_1$ and $y\not=0$. Let $y=y_1u+y_2v$. Independently of the tiling $T$, we know from \Cref{heightvalues} that $h_T(y)=y_1+y_2-3p(y)$, where $p(y)=\frac{\gamma_1y_1+\gamma_2y_2}{n}$. We also know that $h_T(x+y)-h_T(x)=h_T(y)$. Therefore, independently of $T$, the height increment on $x\to\ldots\to x+y$ is $y_1+y_2-3p(y)$. Now, since we are only travelling in the positive directions within $T$, the height increases exactly by $1$ each step. Therefore, the path $x\to\ldots\to x+y$ has exactly $y_1+y_2-3p(y)$ steps. Assume that $A_1$ of them are $u$-arrows, $A_2$ are $v$-arrows and the remaining $A_3=y_1+y_2-3p(y)-A_1-A_2$ are $w$-arrows. Then 
 \begin{align*}
     y_1u+y_2v &=A_1u+A_2v+(y_1+y_2-3p(y)-A_1-A_2)w \\
     &=A_1u+A_2v+(y_1+y_2-3p(y)-A_1-A_2)(-u-v)\\
     &=A_1u+A_2v+(A_1+A_2+3p(y)-y_1-y_2)(u+v)\\
     &=(2A_1+A_2+3p(y)-y_1-y_2)u+(A_1+2A_2+3p(y)-y_1-y_2)v.
 \end{align*}
 Comparing the $u$- and $v$-coefficients, we obtain $2A_1+A_2+3p(y)=2y_1+y_2$ and $A_1+2A_2+3p(y)=y_1+2y_2$. Adding the two equations gives $A_1+A_2=y_1+y_2-2p(y)$ and thus $A_1=y_1-p(y)$, $A_2=y_2-p(y)$ and therefore $A_3=-p(y)$. Now, 
 \begin{align*}
A_1\gamma_1+A_2\gamma_2+A_3\gamma_3&=\gamma_1(y_1-p(y))+\gamma_2(y_2-p(y))+\gamma_3(-p(y))\\
&=\gamma_1y_1+\gamma_2y_2-p(y)(\gamma_1+\gamma_2+\gamma_3) \\
&=\gamma_1y_1+\gamma_2y_2-np(y)=0.
 \end{align*}
 Since $\gamma_1,\gamma_2,\gamma_3>0$ and $A_1,A_2,A_3\ge 0$, we conclude $A_1=A_2=A_3=0$, which implies that the path $x\to\ldots\to x+y$ is empty, in other words, $y=0$, which is a contradiction.
\end{proof}

The following has been proven in more generality in the language of dimer models, see \cite[Theorem 5.6]{Nakajima}. The argument we use below was pointed out to the authors by Fabio Toninelli, see \cite[Lemma 2.6]{Toninelli} for the finite planar version. A similar argument to the planar case was also used in \cite{IyamaOppermann} to describe $2$-APR tilts of $2$-representation finite algebras of type $A$.  

\begin{Theo}\label{Theo: Flips are transitive}
Let $T_1$ and $T_2$ be two $L_1$-periodic tilings of type $(\gamma_1,\gamma_2,\gamma_3)$ with $\gamma_1,\gamma_2,\gamma_3>0$. Then $T_1$ and $T_2$ are connected by a sequence of flips.    
\end{Theo}

\begin{proof}
Let $h(x)=\frac{h_{T_1}(x)-h_{T_2}(x)}{3}$. Every arrow $x\to x+u$ in $L_0$ gives a height increment of either $+1$ or $-2$ depending on whether it belongs to the corresponding tiling. In any case, such an arrow gives a height increment of $1 \bmod 3$ independently of the tiling. One can extend this argument to any path in $L_0$. This shows that $h(x)$ is an integer-valued function. We claim that $h(x)$ is $L_1$-periodic. First note that for every $x\in L_0$ and every $y\in L_1$ we have  $3h(x+y)-3h(x)=(h_{T_1}(x+y)-h_{T_2}(x+y))-(h_{T_1}(x)-h_{T_2}(x))=(h_{T_1}(x+y)-h_{T_1}(x))-(h_{T_2}(x+y)-h_{T_2}(x))=h_{T_1}(y)-h_{T_2}(y)$.
Next, recall from \Cref{heightvalues} that the height of a point in $L_1$ only depends on the \emph{type} of the tiling. Since $T_1$ and $T_2$ have the same type, we conclude that $h_{T_1}(y)=h_{T_2}(y)$. It also follows that $h(L_1) = 0$. 

Since $h(x)$ is $L_1$-periodic and integer-valued, the map $h$ descends to a map from the quotient $\tilde{h}\colon L_0/L_1\to \mathbb{Z}$. Assume that $T_1$ and $T_2$ are different tilings (otherwise we are done), then there is at least one $x_1+L_1\in L_0/L_1$ such that $\tilde{h}(x_1+L_1)\not=0$. Without loss of generality we can assume $\tilde{h}(x_1+L_1)>0$, otherwise we can swap $T_1$ and $T_2$. 
Consider any path $x_1\to\ldots\to x_2$ which lies in $T_1$ (viewed as a digraph on the plane) and is maximal in a sense that there is no $x_3\not=x_2$ such that $x_2\to\ldots\to x_3$ is a path in $T_1$. Note that such a maximal path exists due to \Cref{Lem: Positive type is acyclic}. Also note that it can be empty (in the case of $x_1=x_2$). Clearly, $x_2+L_1$ is a sink of $T_1$. It is also clear that $\tilde{h}(x_2+L_1)\ge \tilde{h}(x_1+L_1)>0$, which also shows that $x_2 + L_1 \neq 0 + L_1$. We perform a flip of $T_1$ at $x_2+L_1$ and obtain a new tiling $T_1'$ of the same type $(\gamma_1,\gamma_2,\gamma_3)$. Now let $h'(x)=\frac{h_{T_1'}(x)-h_{T_2}(x)}{3}$ and let $\tilde{h'}$ be the corresponding map from the quotient. It is again well defined, since $T_1'$ and $T_2$ are of the same type, and integer valued. Note that $\tilde{h'}(x_2+L_1)-\tilde{h}(x_2+L_1)=\frac{h_{T'}(x_2)-h_{T_1}(x_2)}{3}=-1$ and for any $x+L_1\not=x_2+L_1$ we have $\tilde{h'}(x+L_1)=\tilde{h}(x+L_1)$. Therefore, 
$$\sum_{\substack{x+L_1 \in L_0/L_1 \\ \tilde{h'}(x+L_1)>0}}\tilde{h'}(x+L_1)-\sum_{\substack{x+L_1 \in L_0/L_1 \\ \tilde{h}(x+L_1)>0}}\tilde{h}(x+L_1)=-1.$$
Repeating the same process, we will eventually obtain a tiling $T_1''$ and the corresponding map $\tilde{h''}$ such that $\sum_{x+L_1: \tilde{h''}(x+L_1)>0}\tilde{h''}(x+L_1)=0$. In other words, the function $\tilde{h''}$ only takes non-positive values. If $\tilde{h''}$ is the zero map, we are done. Otherwise, we swap the roles of $T_1''$ and $T_2$, changing $\tilde{h''}$ to $-\tilde{h''}$, which only takes non-negative values, and repeat the above process eliminating the positive values.  
\end{proof}

Reformulating the above theorem for cuts gives the following. 

\begin{proof}[Proof of \Cref{Theo: Mutation is transitive}]
    Translate \Cref{Theo: Flips are transitive} via \Cref{Rem: Tilings and Periodic cuts}.
\end{proof}

Let us note that we therefore have that every tiling containing lozenges of all three types can be obtained from a tiling $T_\gamma$ by flips. This has interesting consequences for the representation theory of the $2$-representation infinite algebras, see \Cref{Cor: Tilting equiv to skewed AIR}. 

\begin{Rem}
    In \Cref{Rem: Our tilings are lifted}, we noted that for each nontrivial type $\gamma = (\gamma_1, \gamma_2, \gamma_3)$ of a tiling $T$, we can find a tiling $T_\gamma$ of the same type that is periodic with respect to a lattice $L_2 \supseteq L_1$ such that $L_0/L_2$ is cyclic. The tiling $T_\gamma$ was constructed using an AIR-cut associated to $L_0/L_2$. Together with \Cref{Theo: Flips are transitive}, we therefore see that if $\gamma_1, \gamma_2, \gamma_3 > 0$, then any tiling $T$ of type $\gamma$ is flip-equivalent to $T_\gamma$.     
\end{Rem}

We conclude by reproving \Cref{Theo:SL3 Classification} from the perspective of tilings. This proof is independent of \Cref{Theo:AIR cuts} and \Cref{Theo: LeMeur}. 

\begin{proof}[Proof of \Cref{Theo:SL3 Classification}]
    We use the same notation as in the theorem. It then suffices to show that the conditions from \Cref{Theo: Divisibility conditions} have a solution $(\gamma_1, \gamma_2, \gamma_3)$ with $\gamma_i > 0$ if and only if $G$ is not $C_2 \times C_2$ and the embedding $G \leq \SL_3(k)$ does not factor through an embedding $\SL_2(k) \to \SL_3(k)$. 
    
    First, we consider $G \simeq C_2 \times C_2$. Up to a change of basis for $L_1$, the only matrix $B$ for $G$ is then $B = (\begin{smallmatrix} 2 & 0 \\ 0 & 2\end{smallmatrix})$, and hence we have $4 \mid 2 \gamma_1$ and $4 \mid 2 \gamma_2$, and the only solution with $\gamma_1, \gamma_2 > 0$ is $(2,2,0)$. 

    Next, if the embedding $G \leq \SL_3(k)$ factors through $\SL_2(k) \leq \SL_3(k)$, we choose without loss of generality $ B= (\begin{smallmatrix} 1 & b_1 \\ 0 & n    \end{smallmatrix})$, and we see immediately that $n \mid \gamma_1$ only allows for solutions $\gamma_1 \in \{0,n \}$. 

    Now, assume $G$ is none of the above groups. Then by \Cref{Rem: Positive matrix}, the corresponding matrix can be chosen to be in upper triangular form $B = (\begin{smallmatrix} a_1 & b_1 \\ 0 & b_2    \end{smallmatrix} )$. By assumption, we have $n=a_1b_2$, as well as $a_1 > 1$ and $b_2 > 0$. Furthermore, using the $\SL_2(\mathbb{Z})$-action, we can assume that that $0 \leq b_1 < a_1$. Then, it is easy to check that $(b_2, a_1 - b_1, n - b_2 - a_1 + b_1) = (b_2, a_1 - b_1, a_1b_2 + b_1 - b_2 - a_1)$ is a solution to the conditions from \Cref{Theo: Divisibility conditions}. As long as $n = a_1b_2 > 4$, we have $ a_1b_2 + b_1 - b_2 - a_1 > 0 $, and hence the triple is the type of a higher preprojective cut. In the case $n=4$, we know that $G$ must be cyclic of order $4$, and, after possibly permuting the summands $\rho_1, \rho_2, \rho_3$, a direct computation shows that $B = (\begin{smallmatrix} 4 & 2 \\ 0 & 1    \end{smallmatrix} ) $, hence $ a_1b_2 + b_1 - b_2 - a_1 = 4 + 2 -1 - 4 = 1 > 0 $  In the case $n=a_1 b_2 = 3$, a direct computation shows $B = (\begin{smallmatrix} 3 & 2 \\ 0 & 1    \end{smallmatrix} ) $, and so we have also in this case that $ a_1b_2 + b_1 - b_2 - a_1 = 3 + 2 -1-3 = 1 > 0$. Thus, in all cases, we found the type of a higher preprojective cut.   
\end{proof}
\section*{Acknowledgement}
The authors thank Martin Herschend for many helpful discussions and comments on previous versions of this article. We thank Fabio Toninelli for answering questions regarding his lecture notes and pointing out the proof strategy for \Cref{Theo: Flips are transitive}.  
The second author was supported by a fellowship from the Wenner-Gren Foundations (grant WGF2022-0052).

\printbibliography

\end{document}